\documentclass[12pt]{article}
\usepackage{mystyle}

\newenvironment{simplified}{simplified}

\usepackage[latin1]{inputenc}
\usepackage{amsmath,amsthm,amssymb, bbm, tikz, verbatim, mathrsfs,bbm} \usepackage{comment}
\usepackage{graphicx}
\usepackage{framed}

\graphicspath{{./images/}}

\usepackage{pdfsync}
\usepackage{epsfig}
\usepackage{epstopdf}
\usepackage{url}

\usepackage{ifb}

\graphicspath{{./images/}}
 
\title{Geodesics on Shape Spaces
with \\ Bounded Variation and Sobolev Metrics}

\author{
	\begin{tabular}{c}
		Giacomo Nardi, Gabriel Peyr\'e,  Fran\c{c}ois-Xavier Vialard \\[2mm]
		Ceremade, Universit\'e Paris-Dauphine\footnote{\texttt{\{nardi,peyre,vialard\}@ceremade.dauphine.fr}}\\
	\end{tabular}	
}
\date{}

\begin{document}

\maketitle
\excludecomment{simplified}


\begin{abstract}
This paper studies the space of $BV^2$ planar curves endowed with the $BV^2$ Finsler metric over its tangent space of displacement vector fields. Such a space is of interest for applications in image processing and computer vision because it enables piecewise regular curves that undergo piecewise regular deformations, such as articulations. The main contribution of this paper is the proof of the existence of the shortest path between any two $BV^2$-curves for this Finsler metric.
Such a result is proved by applying the direct method of calculus of variations to minimize the geodesic energy.
This method applies more generally to similar cases such as the space of curves with $H^k$ metrics for $k\geq 2$ integer. This space has a strong Riemannian structure and is geodesically complete. Thus, our result shows that the exponential map is surjective, which is complementary to geodesic completeness in infinite dimensions.
We propose a finite element discretization of the minimal geodesic problem, and use a gradient descent method to compute a stationary point of the  energy. Numerical illustrations show the qualitative difference between $BV^2$ and $H^2$ geodesics.


\end{abstract}

\section{Introduction}\label{intro}

This paper addresses the problem of the existence of minimal geodesics in spaces of planar curves endowed with several metrics over the tangent spaces. Given two initial curves, we prove the existence of a minimizing geodesic joining them. Such a result is proved by the direct method of calculus of variations.

We treat the case of  $BV^2$-curves and $H^k$-curves ($k\geq 2$ integer). Although the proofs' strategies are the same, the $BV^2$ and $H^k$ cases are slightly different and the proof in the $H^k$ case is simpler. This difference is essentially due to the inherent geometric structures (Riemannian or Finslerian) of each space. 

We also propose a finite element discretization of the minimal geodesic problem. We further relax the problem to obtain a smooth non-convex minimization problem. This enables the use of a gradient descent algorithm to compute a stationary point of the corresponding functional. Although these stationary points are not in general global minimizers of the energy, they can be used to numerically explore  the geometry of the corresponding spaces of curves, and to illustrate the differences between the Sobolev and $BV^2$ metrics.

\subsection{Previous Works}
\label{sec-previous-works}

\paragraph{Shape spaces as Riemannian spaces.}

The  mathematical study of spaces of curves has been largely investigated in recent years; see, for instance, ~\cite{Younes-elastic-distance,Mennucci-CIME}. 
The set of curves is naturally modeled over a Riemannian manifold~\cite{MaMi}. This consists in defining  a Hilbertian metric on each tangent plane of the space of curves, i.e. the set of vector fields which deform infinitesimally a given curve. 
Several recent works~\cite{MaMi,charpiat-new-metrics,Yezzi-Menn-2005a,Yezzi-Menn-2005b} point out that  the choice of the metric notably affects the results of gradient descent algorithms for the numerical minimization of functionals. Carefully designing the metric is therefore crucial to reach better local minima of the energy and also to compute descent flows with specific behaviors. These issues are crucial for applications in image processing (e.g. image segmentation) and computer vision (e.g. shape registration).
Typical examples of such  Riemannian metrics are Sobolev-type metrics~\cite{sundaramoorthi-sobolev-active,sundaramoorthi-2006,sundaramoorthi-new-possibilities,Yezzi-H2}, which lead to smooth curve evolutions. 

\paragraph{Shape spaces as Finslerian spaces.}

It is possible to extend this Riemannian framework by considering more general metrics on the tangent planes of the space of curves. Finsler spaces make use of Banach norms instead of Hilbertian norms~\cite{RF-geometry}. A few recent works~\cite{Mennucci-CIME,Yezzi-Menn-2005a,Rigid-evol} have studied the theoretical properties of Finslerian spaces of curves. 

Finsler metrics are used in~\cite{Rigid-evol} to perform curve evolution in the space of $BV^2$-curves. The authors make use of a generalized gradient, which is the steepest descent direction  according to the Finsler metric. The corresponding gradient flow enables piecewise regular evolutions (i.e. every intermediate curve is piecewise regular), which is useful for applications such as registration of articulated shapes.  
The present work naturally follows~\cite{Rigid-evol}. Instead of considering gradient flows to minimize smooth functionals, we consider the minimal geodesic problem. 
However, we do not consider the Finsler metric favoring piecewise-rigid motion, but instead the standard $BV^2$-metric.
In \cite{Bredies}, the authors study a functional space similar to $BV^2$ by considering functions with finite  total  generalized variation. However, such a framework is not adapted to our applications because functions with finite total generalized variation can be discontinuous.

Our main goal in this work is to study the existence of solutions, which is important to understand the underlying space of curves. This is the first step towards a numerical solution to the minimal path length problem for a metric that favors piecewise-rigid motion.


\paragraph{Geodesics in shape spaces.}

The computation of geodesics over Riemannian spaces is now routinely used in many imaging applications. Typical examples of applications include shape registration~\cite{Mennucci-filtering,Younes-04,2010.03.20}, tracking~\cite{Mennucci-filtering}, and shape deformation~\cite{Kilian-shape-space}. In \cite{Rumpf}, the authors study  discrete geodesics and their relationship with continuous geodesics in the Riemannian framework.  Geodesic computations also serve as the basis to perform statistics on shape spaces (see, for instance,~\cite{Younes-04,Arsigny06}) and to generalize standard tools from Euclidean geometry such as averages~\cite{Arsigny:Siam:07}, linear regression~\cite{Niethammer_GeodReg}, and cubic splines~\cite{2010.03.20}, to name a few.  
However, due to the infinite dimensional nature of shape spaces, not all Riemannian structures lead to well-posed length-minimizing problems. For instance, a striking result~\cite{MaMi,Yezzi-Menn-2005b,Yezzi-Menn-2005a} is that the natural $L^2$-metric on the space of curves is degenerate, despite its widespread use in computer vision applications. Indeed, the geodesic distance between any pair of curves is equal to zero. 

The study of the geodesic distance over shape spaces (modeled as curves, surfaces, or diffeomorphisms) has been widely developed in the past ten years~\cite{MR3132089,MR3080480,MR3011892}. We refer the reader to~\cite{2013arXiv1305.1150B} for a review of this field of research. These authors typically address the questions of existence of the exponential map, geodesic completeness (the exponential map is defined for all time), and  the computation of the curvature. In some situations of interest, the shape space has a strong Riemannian metric (i.e., the  inner product on the tangent space induces an isomorphism between the tangent space and its corresponding cotangent space) so that the exponential map is a local diffeomorphism. 
In~\cite{MuMi_ov} the authors describe geodesic equations for Sobolev metrics. They show in Section 4.3 the local existence and uniqueness of a geodesic with prescribed initial conditions. This result is improved in~\cite{Bruveris}, where the authors prove the existence for all time. Both previous results are proved by techniques from ordinary differential equations.
In contrast, local existence (and uniqueness) of minimizing geodesics with prescribed boundary conditions (i.e. between a pair of curves) is typically obtained using the exponential map. 

In finite dimensions, existence of minimizing geodesics between any two points (global existence) is obtained by the Hopf-Rinow theorem~\cite{HR}. Indeed, if the exponential map is defined for all time (i.e. the space is geodesically complete) then global existence holds. This is, however, not true in infinite dimensions, and a counterexample of non-existence of a geodesic between two points over a manifold is given in~\cite{MR0188943}. An even more pathological case is described in~\cite{MR0400283}, where an example is given where the exponential map is not surjective although the manifold is geodesically complete.
Some positive results exist for infinite dimensional manifolds (see in particular Theorem B in~\cite{Ekland} and Theorem 1.3.36 in~\cite{Mennucci-CIME}) but the surjectivity of the exponential map still needs to be checked directly on a case-by-case basis. 

In the case of a Finsler structure on the shape space, the situation is more complicated, since the norm over the tangent plane is often non-differentiable . This non-differentiability is indeed crucial to deal with curves and evolutions that are not smooth (we mean evolutions of non-smooth curves). That implies that geodesic equations need to be understood in a weak sense. More precisely, the minimal geodesic problem can be seen as a Bolza problem on the trajectories $H^1([0,1],BV^2(\Circ,\RR^2))$. In~\cite{Mord} several necessary conditions for existence of solutions to Bolza problems in Banach spaces are proved within the framework of differential inclusions. Unfortunately, these results require some hypotheses on the Banach space (for instance the Radon-Nikodym property for the dual space) that are not satisfied by the Banach space that we consider in this paper. 

We therefore tackle these issues in the present work and prove  existence of minimal geodesics in the space of $BV^2$ curves by a variational approach. We also show how similar techniques can be applied to the case of Sobolev metrics.

\subsection{Contributions}

Section~\ref{BV2} deals with the Finsler space of $BV^2$-curves. Our main contribution is Theorem~\ref{local_existence}  proving the existence of a minimizing geodesic between two $BV^2$-curves. We also explain how this result can be generalized to the setting of geometric curves (i.e. up to reparameterizations).

Section~\ref{Hs} extends these results to $H^k$-curves with $k\geq 2$ integer, which gives rise to Theorems~\ref{existence_sobolev} and \ref{geom-sob}.  Our results are complementary to those presented in~\cite{MuMi_ov} and~\cite{Bruveris} where the authors show the geodesic completeness of curves endowed with the $H^k$-metrics with $k \geq 2$ integer. We indeed show that the exponential map is surjective.

Section~\ref{discretization} proposes a discretized minimal geodesic problem for $BV^2$ and Sobolev curves. We show  numerical simulations for the computation of stationary points of the energy. In particular, minimization is made by a gradient descent scheme, which requires, in the $BV^2$-case, a preliminary regularization of the geodesic energy.  


\section{Geodesics in the Space of $BV^2$-Curves}\label{BV2}

In this section we define  the set of parameterized $BV^2$-immersed curves and we prove several useful properties. In particular, in Section \ref{repa}, we discuss the properties of reparameterizations of $BV^2$-curves.

The space of parameterized $BV^2$-immersed curves can be  modeled as a Finsler manifold as presented in Section \ref{tangent}.  Then, we can define  a geodesic Finsler distance and prove the existence of a geodesic between two $BV^2$-curves  (Sections~\ref{existence_geod}). Finally, we define the space of geometric curves (i.e., up to reparameterization) and we prove similar results (Section~\ref{13}).
We point out that, in both the parametric and the geometric case,  the geodesic is not unique in general.
Through  this paper we identify the circle $\Circ$ with $[0,1] / \{0 \sim 1\}$.

\subsection{The Space of $BV^2$-Immersed Curves}
\label{11}
 
Let us first recall some needed defintions.
\begin{defn}[{\bf $BV^2$-functions}]
We say that $f\in L^1(\Circ,\RR^2)$ is a function of bounded variation if its first variation $|D f|(\Circ)$ is finite:
$$	 |D f|(\Circ) = \sup \enscond{
		\int_{\Circ}f(s)\cdot g'(s)\, \d s 
	}{
		g\in \mathrm{C}^\infty(\Circ, \R^2),\|g\|_{L^\infty(\Circ,\R)}\leq 1
	} <\infty\,.
$$
 Several times in the following, we use  the fact that the space of functions of bounded variation is a Banach algebra and a chain rule holds. We refer to \cite[Theorem 3.96, p. 189]{AFP} for a proof of these results.

We say that $f\in \BV^2(\Circ,\R^2)$ if $f\in W^{1,1}(\Circ,\R^2)$ and its second variation $|D^2 f|(\Circ)$  is finite:

$$	 |D^2 f|(\Circ) = \sup \enscond{
		\int_{\Circ}f(s)\cdot g''(s)\, \d s 
	}{
		g\in \mathrm{C}^\infty(\Circ, \R^2),\|g\|_{L^\infty(\Circ,\R)}\leq 1
	} <\infty\,.
$$
For a sake of clarity we point out that, as $W^{1,1}\subset BV$, for every $BV^2$-function, the  first variation  coincides  with the $L^1$-norm of the derivative. Moreover, by integration by parts, it holds
$$|D^2 f|(\Circ)=|D f'|(\Circ)\,.$$
\end{defn}

The $BV^2$-norm is defined as
$$\|f\|_{\BVd} = \|f\|_{W^{1,1}(\Circ,\R^2)} + |D^2 f|(\Circ)\,.$$

The space $BV^2(\Circ,\R^2)$ can also be equipped with the following types of convergence, both weaker than the norm convergence:

\begin{itemize}
\item[1.] {\em  Weak* topology}. Let $\{f_h\}\subset BV^2(\Circ,\R^2)$ and $f\in BV^2(\Circ,\R^2)$. We say that $\{f_h\}$  weakly* converges in $BV^2(\Circ,\R^2)$ to $f$ if
$$f_h \overset{W^{1,1}(\Circ,\R^2)}{\longrightarrow} f \quad\mbox{and}\quad D^2 f_h \overset{*}{\rightharpoonup}D^2 f\,, \quad\mbox{as}\quad h \rightarrow \infty\,,$$
where $\overset{*}{\rightharpoonup}$ denotes the weak* convergence  of measures.
\item[2.] {\em Strict topology}. Let $\{f_h\}\subset BV^2(\Circ,\R^2)$ and $f\in BV^2(\Circ,\R^2)$. We say that $\{f_h\}$  strictly converges to $f$ in $BV^2(\Circ,\R^2)$ if
$$f_h \overset{W^{1,1}(\Circ,\R^2)}{\longrightarrow} f \quad\mbox{and}\quad |D^2 f_h|(\Circ) \longrightarrow |D^2 f|(\Circ)\,,\quad\mbox{as}\quad h \rightarrow \infty.$$
\par Note that the following distance 
$$d(f,g) = \|f-g\|_{L^1(\Circ,\R^2)} +||D^2f|(\Circ)-|D^2g|(\Circ)|$$
is a distance in $BV^2(\Circ,\R^2)$ inducing the strict convergence. 
\end{itemize}

The following results can be deduced by the theory of functions of  bounded variation \cite{AFP, EG}.

\begin{prop}[{\bf weak* convergence}]
Let $\{f_h\}\subset BV^2(\Circ,\R^2)$. Then $\{f_h\}$  weakly* converges  to $f$ in $BV^2(\Circ,\R^2)$ if and only if  $\{f_h\}$ is bounded in $BV^2(\Circ,\R^2)$ and strongly converges to $f$  in $W^{1,1}(\Circ,\R^2)$.
\end{prop}

\begin{prop}[{\bf embedding}]
The following continuous embeddings hold:
$$\BVd \hookrightarrow W^{1,\infty}(\Circ,\R^2)\;,\quad \quad \BVd \hookrightarrow C^0(\Circ,\R^2)\,.$$
In particular (see Claim 3 p.218 in~\cite{EG}) we have 
\begin{equation}\label{embedding}
	\foralls f\in BV(\Circ,\R^2), \quad
	\norm{f}_{L^\infty(\Circ,\R^2)}\leq   \norm{f}_{BV(\Circ,\R^2)}.
\end{equation}
\end{prop}
We refer to  \cite{Maitine-BV2} for a deeper analysis of   $BV^2$-functions. \newline \\

We can now define the set of $BV^2$-immersed curves and prove that it is a manifold modeled on $\BVd$. In the following we denote by $\gm$ a generic $BV^2$-curve and by $\gm'$ its derivative.
Recall also that, as $\gm'$ is a $BV$-function of one variable, it admits a left and right limit at every point of $\Circ$ and it is continuous everywhere except on a (at most) countable set of  points of $\Circ$.
The space of smooth immersion of $\Circ$ is defined by
\begin{equation}\label{immersion-smooth}
	\Imm(\Circ,\R^2)\, = \,
	\enscond{ \gm \in C^{\infty}(\Circ,\R^2)  }{
		\gm'(s) \neq 0 \quad \foralls \, s \in \Circ }\,.
\end{equation} 
The natural extension of this definition to $\BVd$-curves is 
\begin{equation}\label{immersion-BV2}
	\Imm_{\BVd}(\Circ,\R^2)\, = \,
	\enscond{	\gm \in \BVd  }{
		0 \notin [\lim_{t\to s^+} \gm'(t),\lim_{t\to s^-} \gm'(t)] \quad \foralls \, s \in \Circ }\,,
\end{equation} 
where $[\lim_{t\to s^+} \gm'(t),\lim_{t\to s^-} \gm'(t)]$ denotes the segment connecting the two points.
This definition implies that $\gm$ is locally the graph of a $BV^2(\R,\R)$ function. However, in the rest of the paper, we relax this assumption and work on a larger space under the following definition.

\begin{defn}
[{\bf $BV^2$-immersed curves}]\label{BV2curves}
A $BV^2$-immersed curve is any closed curve  $\gm \in BV^2( \Circ, \R^2)$ satisfying
\begin{equation}\label{cond-derivative}
	\underset{t\rightarrow s^{+(-)}}{\lim} \|\gm'(t)\| \neq 0	\quad \forall s\in\Circ\,.
\end{equation}
We denote by $\Bb$ the set of $BV^2$-immersed curves.  
\end{defn}
Although a bit confusing, we preferred to work with this definition of immersed curves, since it is a stable subset of $\BVd$ under reparameterizations.
Note that this definition allows for cusp points and thus curves in $\Bb$ cannot be in general viewed as the graph of a $BV^2(\R,\R)$ function.
Condition \eqref{cond-derivative} allows one to define a Fr\'enet frame for a.e.-$s\in \Circ$ by setting
\begin{equation}\label{def-tangente}
	{\bf t}_{\gm}(s) = \frac{\gm'(s)}{\|\gm'(s)\|}\,,\quad \quad \n_\gm(s)= {\bf t}_{\gm}(s)^\bot\,,
\end{equation}
\eq{
	\qwhereq (x,y)^\bot= (-y,x)\quad \foralls (x,y)\in \RR^2.
}
Finally we denote by $\len(\gm)$  the length of $\gm$ defined as
\begin{equation}\label{length-def}
\len(\gm)=\int_{\Circ}\|\gm'(s)\|\d s\,.
\end{equation}

The next proposition proves a useful equivalent property of \eqref{cond-derivative}:
\begin{prop}  Every  $\gm\in BV^2(\Circ,\RR^2)$ satisfies \eqref{cond-derivative} if and only if 
\begin{equation}\label{cond-derivative-bis}
\underset{s\in \Circ}{\mbox{essinf}} \,\|\gm'(s)\|\,>\,0\,.
\end{equation}

\end{prop}

\begin{proof} As $\gm'\in BV(\Circ, \RR^2)$, it admits a left and right limit at every point of $\Circ$ so that we can define the following functions:
$$\forall\,s\in \Circ\;,\quad \gm_l'(s) = \underset{t \rightarrow s^-}{\lim}\,\gm'(t)\;,\quad \quad  \gm_r'(s) = \underset{t \rightarrow s^+}{\lim}\,\gm'(t)\,,$$
where $\gm'_l$ and $\gm'_r$ are continuous from the left and the right, respectively and satisfy 
$$\underset{s\in \Circ}{\mbox{essinf}} \,\|\gm'(s)\|= \underset{s\in \Circ}{\mbox{essinf}} \,\|\gm_r'(s)\|= \underset{s\in \Circ}{\mbox{essinf}} \,\|\gm_l'(s)\|\,.$$

Let us suppose that  $\gm'$ verifies \eqref{cond-derivative} and $\underset{s\in \Circ}{\mbox{essinf}} \,\|\gm'(s)\|\,=\,0$. Then we can define a sequence $\{s_n\}\subset \Circ$ such that $\gm'(s_n)\to 0$, and (up to a subsequence) we have $s_n\to s$ for some $s\in \Circ$. Now, up to a subsequence, the sequence $s_n$ is a left-convergent sequence (or a right-convergent sequence), which implies that $\gm'_l(s_n)\to \gm_l'(s) =0$. This is of course in contradiction with \eqref{cond-derivative}. The right-convergence case is similar.

Now, let us suppose that $\gm'$ satisfies \eqref{cond-derivative-bis} so that $\gm'_r$ and $\gm'_l$ also satisfy \eqref{cond-derivative-bis}. 
Then if $ \underset{t \rightarrow s^-}{\lim}\,\gm'(t) =0$ for some $s\in \Circ$, for  every $\varepsilon <  \underset{s\in \Circ}{\mbox{essinf}} \,\|\gm_l'(s)\|$ there exists $\delta$ such that $]s-\delta,s]\subset \{\|\gm'_l\|<\varepsilon\}$, which is in contradiction  with $\underset{s\in \Circ}{\mbox{essinf}} \,\|\gm_l'(s)\|>0$. This proves that the left limit is positive at every point. By using $\gm'_r$ we can similarly show that the right limit is also positive, which proves \eqref{cond-derivative}.
\end{proof}
We can now show that $\Bb$ is  a manifold modeled on $\BVd$ since it is open in $\BVd$.

\begin{prop}\label{openB0}   
$\Ba$ is an open set  of $\BVd$. 
\end{prop}

\begin{proof}

Let $\gm_0\in \Ba$. We prove that 
\begin{equation}\label{ball}
U_{\gm_0} =		\enscond{ \gm\in \BVd 
		}{
			\norm{\gm - \gm_0}_{BV^2(\Circ,\RR^2)} \leq \frac{1}{2} \underset{s\in\Circ}{\mbox{essinf}}\,\|\gm'(s)\|
		}
		\subset \Ba\,.
\end{equation}
In fact, by~\eqref{embedding}, we have $\norm{\gm'}_{L^\infty(\Circ,\R^2)}\leq  \norm{\gm'}_{\BV(\Circ,\R^2)}$,  so that  every curve $\gm\in BV^2(\Circ,\R^2)$ such that  
$$\|\gm - \gm_0\|_{BV^2(\Circ,\R^2)}\leq\frac{1}{2} \underset{s\in\Circ}{\mbox{essinf}}\,\|\gm'(s)\|\;$$
satisfies~\eqref{cond-derivative-bis}.  

\end{proof}

\begin{rem}[{\bf immersions, embeddings, and orientation}]\label{properties}

We point out that condition \eqref{cond-derivative} does not guarantee that curves belonging to  $\Bb$ are injective. 
This implies in particular that every element of  $\Bb$ needs not  be  an embedding (see Figure~\ref{not-emb}).

\begin{figure}[h!]
\centering
\includegraphics[width=.4\linewidth, angle=270]{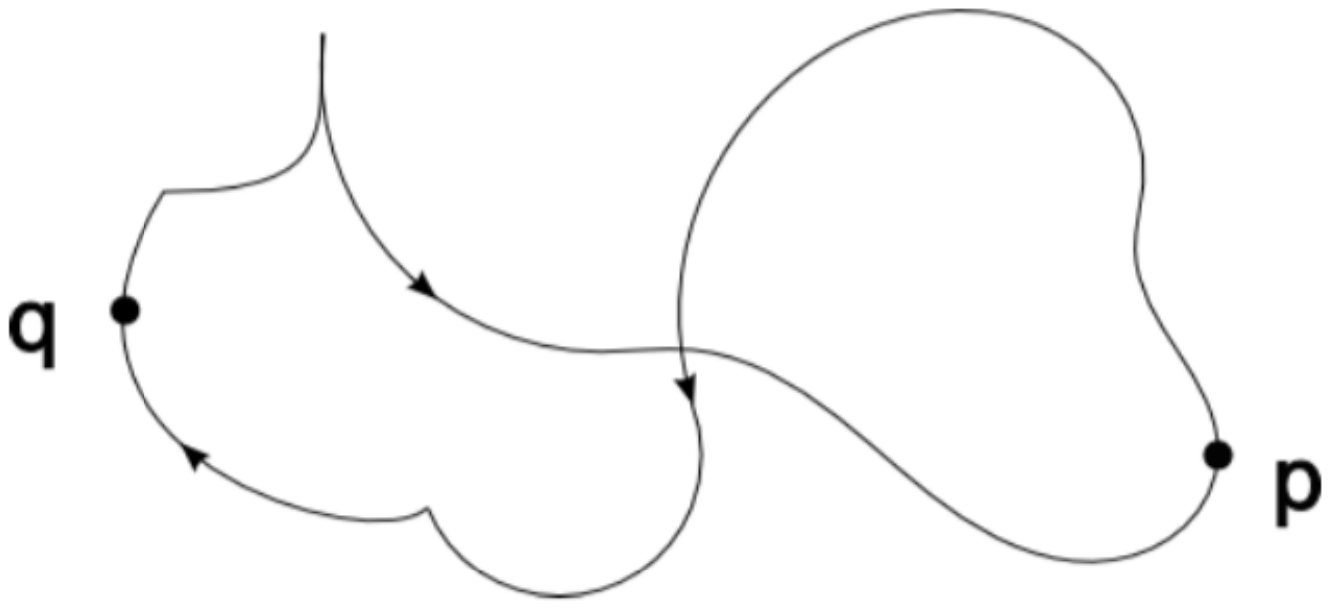}
\caption{\label{not-emb} Non-injective $BV^2$-immersed curve positively oriented with respect to $p$.}
\end{figure}

Moreover, as $BV^2$-immersed curves can have some self-intersections,  the standard notion of orientation (clockwise or counterclockwise) defined for Jordan's curves cannot be used in our case. The interior of a $BV^2$-immersed curve can be disconnected and the different branches of the curve can be parameterized with incompatible orientations. For example, there is no  standard counterclockwise parameterization of the curve in Fig.\ref{not-emb}.

In order to define a suitable notion of orientation, we introduce the notion of orientation with respect to an extremal point. For every $\gm\in \Bb$ and $p\in \gm(\Circ)$ we say that $p$ is an extremal point for $\gm$ if $\gm(\Circ)$ lies  entirely in a closed half-plane bounded by a line
through $p$. 

We also suppose that the Fr\'{e}net frame denoted $({\bf t}_p, {\bf n}_p)$ is well defined at $p$, where ${\bf n}_p$ denotes here the unit outward normal vector. Then, we say that $\gm$ is positively oriented with respect to $p$ if  the ordered pair $({\bf n}_p, {\bf t}_p)$ gives the
counterclockwise orientation of $\R^2$. 
For example the curve in  Fig.\ref{not-emb} is positively oriented with respect to the point $p$ but negatively oriented with respect to $q$.

\end{rem}

\subsection{Reparameterization of $BV^2$-Immersed Curves}\label{repa}

In this section we introduce the set of reparameterizations adapted to  our setting. We prove in particular that it is always possible to define a constant speed reparameterization.

Moreover, we point out  several properties describing the relationship between the convergence of parameterizations and the convergence of the reparameterized curves.
On one hand, in Remark \ref{disc-metric} we underline that the reparameterization operation is not continuous with respect to the $BV^2$-norm. On the other hand, Lemma \ref{conv-w11} proves that  the convergence of the curves implies the convergence of the respective constant speed parameterizations.

\begin{defn}[{\bf reparameterizations}]\label{reparam}
We denote by 	${\rm Diff}^{BV^2}(\Circ)$  the set of homeomorphisms  $\varphi\in BV^2(\Circ,\Circ)$ such that $\varphi^{-1} \in BV^2(\Circ,\Circ)$. The elements of  ${\rm Diff}^{BV^2}(\Circ)$ are called reparameterizations.
Note that any $\varphi \in {\rm Diff}^{BV^2}(\Circ)$ can be considered as an element of $BV_{loc}^2(\RR,\RR)$ by the lift operation (see \cite{ghys}). Moreover the usual topologies (strong, weak, weak*)
on subsets of ${BV^2}(\Circ,\Circ)$ will be induced by the standard topologies on the corresponding subsets of ${BV^2_{loc}}(\RR,\RR)$.

\end{defn}

The behavior of $\BVd$ curves under reparameterizations is discontinuous due to the strong $BV$ topology as described below.
\begin{rem}[{\bf discontinuity of the  reparameterization operation}]\label{disc-metric}

In this remark we give a counterexample to the following conjecture: for every $\gm\in \Bb$ and for every  sequence of parameterizations $\{\varphi_h \}$ strongly converging  to $\varphi\in BV^2(\Circ, \Circ)$ we obtain that $\gm\circ\varphi_h$ strongly converges  to $\gm\circ\varphi$  in $BV^2(\Circ, \RR^2)$ .

This actually proves  that the composition with a reparameterization is not a continuous function from the set of  reparameterizations to $\Bb$.

We consider the curve $\gm$ drawn in Fig. \ref{corner} and we suppose that it is  counterclockwise oriented and that the corner point corresponds to the parameter $s=0$. Note also that the second variation of $\gm$ is represented by a Dirac delta measure $v\delta_0$ in a neighborhood of $s=0$, where $v$ is a vector such that $\|v\|> \alpha >0$.

Then we consider the family of parameterizations defined by

$$ \quad \varphi_h(s)=s+\frac 1 h\,,$$
where the addition is considered modulo $1$. This 
sequence of reparameterizations shifts the corner  point on $\Circ$ and converges $BV^2$-strongly to the identity reparameterization $\varphi(s)=s$. 

Moreover, we have that 
$\gm(\varphi_h(s))=\gm(s+1/h)$ for every $s\in\Circ$, which implies that $\gm\circ\varphi_h$ converges to $\gm\circ\varphi=\gm$ strongly in $W^{1,1}(\Circ,\RR^2)$.
However, similarly to $\gm$,  the second variation  of $\gm\circ\varphi_h$  is represented by a Dirac delta measure in a neighborhood of the parameter corresponding to the corner.  Then 
$$|D (\gm\circ\varphi_h-\gm)|(\Circ)\,>\,\alpha\;,\ \quad\quad\varepsilon \rightarrow +\infty\,,$$ which implies that  the reparameterized curves do not converge to the initial one with respect the $BV^2$-strong topology.

\begin{figure}[h!]
\centering
\includegraphics[width=.3\linewidth]{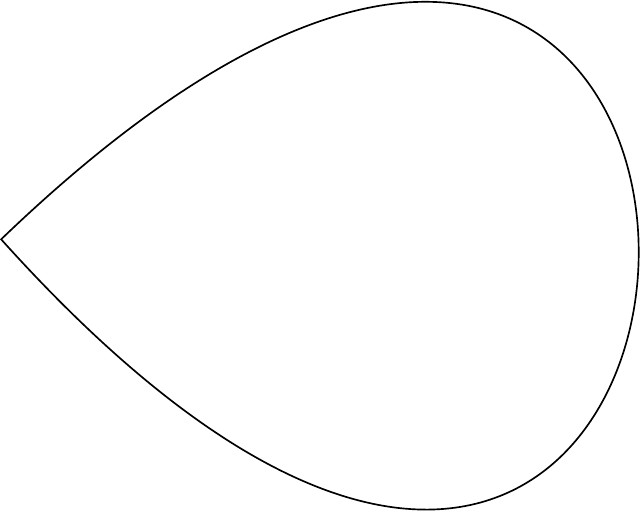}
\caption{\label{corner} Immersed $BV^2$-curve with a corner.}\vspace{1cm}
\end{figure}

\end{rem}

\begin{rem}[{\bf constant speed parameterization}]\label{arc-len} Property~\eqref{cond-derivative}  allows us to define the constant speed parameterization for every $\gm\in \Bb$. We start by setting
$$s_\gm : \Circ \rightarrow \Circ\,,$$
$$s_\gm(s) =\frac{1}{\len(\gm)} \int_{s_0}^s \,\|\gm'(t)\| \, \d t \;,\quad s_0\in \Circ\,$$
where $\len(\gm)$ denotes the length of $\gm$ defined in \eqref{length-def} and where $s_0$ is a chosen basepoint belonging to $\Circ$. 
Now, because of \eqref{cond-derivative}, we can define $\varphi_\gm =s_\gm^{-1}$ and  the constant speed parameterization of $\gm$ is given by $\gm\circ \varphi_\gm$. 
In order to prove that $s_{\gm}$ is invertible we apply the result proved in \cite{Clarke}. In this paper the author gives a condition on the generalized derivative of a Lipschitz-continuous function in order to prove that it is invertible. We detail how to apply this result to our case.

Because of Rademacher's theorem, as $s_{\gm}$ is Lipschitz-continuous, it is a.e. differentiable. Then we consider  the generalized derivative at $s\in \Circ$, which is defined as the convex hull of the elements $m$ of the form
$$m\,=\, \underset{i\rightarrow +\infty}{\lim} s_{\gm}'(s_i)\,,$$
where $s_i\rightarrow s$ as $i\rightarrow +\infty$ and  $s_{\gm}$ is differentiable at every $s_i$.
Such a set is denoted by $\partial s_{\gm}(s)$ and it is a non-empty compact convex set of $\R$. Now, in \cite{Clarke} it is proved that if $0\notin \partial s_{\gm}(s)$ then $s_{\gm}$ is locally invertible at $s$. We remark that, in our setting, such a condition is satisfied because of \eqref{cond-derivative} so that the constant speed parameterization is well defined for every $\gm\in \Bb$.

Finally we remark that $s_{\gm}, \varphi_\gm\in BV^2(\Circ, \Circ)$. For a rigorous proof of this fact we refer to Lemma~\ref{bound-rep-bv}.
\end{rem}

The next two lemmas prove some useful properties of the constant speed parameterization.

\begin{lem}\label{bound-rep-bv}
If $\gm \in \Bb$ is such that $\underset{s\in \Circ}{\mbox{essinf}}\,\| \gm'(s) \| \geq \varepsilon >0$ and $\| \gm \|_{BV^2(\Circ,\R^2)} \leq M$, then  there exists a positive constant $D = D(\varepsilon,M)$ such that 
$ \| \phi_\gm \|_{BV^2(\Circ,\Circ)} \leq D$. 
\end{lem}

\begin{proof}
Recall that the reparameterization $\phi_\gm$ is the inverse of  $s_\gm(s) =\frac{1}{\len(\gm)} \int_{s_0}^s \,\|\gm'(t)\| \, \d t$,
where $s_0$ is a chosen basepoint belonging to $\Circ$. Then in particular we have  
$$ \underset{s\in \Circ}{\mbox{essinf}}\,s_\gm'(s) \geq \frac{\varepsilon }{ \len(\gm)} \,.$$ Moreover, because of \eqref{embedding}, $\| \gm' \|$ is bounded by $M$. In particular we have $\| s_\gm \|_{L^\infty(\Circ, \Circ)}\leq 1$ and $\| s_\gm' \|_{L^1(\Circ, \Circ)}=1$, and, by the chain rule for $BV$-functions, we also get $|D s_\gm'|(\Circ)\leq \beta M$ with $\beta=\beta(\varepsilon, \len(\gamma))$. 
We finally have
$$\| s_\gm \|_{BV^2(\Circ, \Circ)}\leq 2(1+\beta M).$$
Then, by a straightforward calculation and the chain rule, we get that 

$$\| \varphi_\gm \|_{L^1(\Circ, \Circ)}\leq 1 \;, \quad \| \varphi_\gm' \|_{L^{\infty}(\Circ, \Circ)}\leq  \frac{\len(\gm)}{\varepsilon} \;, \quad  |D\varphi_\gm'|(\Circ)\leq \frac{\len(\gm)^2}{\varepsilon^2}|Ds_\gm'|(\Circ)\,,$$
which proves the lemma.
\end{proof}

\begin{lem}\label{conv-w11}
Let $\{\gm_h\}\subset \Ba$ be a sequence satisfying  
\begin{equation}\label{hyp-conv}
0<\underset{h}{\inf}\,\underset{s\in \Circ}{\mbox{essinf}}\,\| \gm_h'(s) \| < \underset{h}{\sup}\,\| \gm_h' \|_{L^\infty(\Circ,\R^2)} < \infty \,,\quad \underset{h}{\inf}\,\len( \gm_h) > 0\,,
\end{equation}
and  converging to $\gm \in \Ba$ in $W^{1,1}(\Circ,\RR^2)$. Then $\varphi_{\gm_h}\rightarrow \varphi_{\gm}$ in $W^{1,1}(\Circ, \Circ)$.
\end{lem}

\begin{proof}
By \eqref{hyp-conv} and the dominated convergence theorem we can prove that    $s_{\gm_h}\rightarrow s_\gm$ in $W^{1,1}(\Circ, \Circ)$.
Moreover, because of Lemma \ref{bound-rep-bv},  $\varphi_\gm \in BV^2(\Circ, \Circ)$, so it is continuous. 
Now, by performing the change of variable $s= s_{\gm_h}(t)$, we get 
$$\int_{\Circ} \|\varphi_{\gm_{h}}(s)-\varphi_{\gm}(s)\| \d s\,=\, \int_{\Circ} \|t-\varphi_{\gm}(s_{\gm_h}(t))\|\,\| s_{\gm_h}'(t)\| \d t\,,$$
$$\int_{\Circ} \|\varphi_{\gm_{h}}'(s)-\varphi_{\gm}'(s)\| \d s\,=\, \int_{\Circ} \left\|\frac{1}{s_{\gm_h}'(t)}-\frac{1}{s_\gm'(\varphi_{\gm}(s_{\gm_h}(t)))}\right\|\,\| s_{\gm_h}'(t)\| \d t\,.$$
Then, as $\varphi_\gm$ is continuous and  $s_{\gm_h}\rightarrow s_\gm$ in $W^{1,1}(\Circ, \Circ)$, we get the result by \eqref{hyp-conv} and  the dominated convergence theorem. 
\end{proof}

\subsection{The Norm on the Tangent Space}\label{tangent}

We can now define the norm on the tangent space to $\Bb$ at $\gm\in \Bb$, which is used to define the length of a path. We first recall the main definitions and properties of functional spaces equipped with the measure $\d \gm$.
 
\begin{defn}[\bf functional spaces w.r.t. $\d \gm$]\label{dgm}

Let $\gm\in\Bb$ and $f:\Circ\rightarrow \RR^2$. 
We consider the following  measure $\d\gm$ defined as 
$$
\d \gm(A)=\int_A \|\gm'(s)\| \d s \; \quad \forall \, A\subset \Circ\,.
$$
Note that, as 
$\int_A \d s =0 \;\; \Leftrightarrow \;\; \int_A \d\gm(s)=0$
for every open set $A$ of the circle, we get 
\begin{equation}\label{eq-equi-norm-linf}
	\norm{f}_{L^\infty(S^1,\RR^2)} = \norm{f}_{L^\infty(\gm)}\,,
\end{equation} 
where 
$$\norm{f}_{L^\infty(\gm)} = \inf\,\{a\,:\, f(x)<a\quad d\gm-a.e.\}\,. $$
Moreover, the derivative and the $L^1$-norm with respect to such a measure are given by
$$\frac{\d f}{\d \gm}(s)\,=\, \underset{\varepsilon \rightarrow 0}{\lim} \, \frac{f(s+\varepsilon)-f(s)}{\d \gm((s-\varepsilon, s+\varepsilon))} =\frac{f'(s)}{\|\gm'(s)\|} \,,\quad \|f\|_{L^1(\gm)} = \int_{\Circ} \|f(s)\|\|\gm'(s)\|\, \d s\,. $$
Note that, as $\gm\in \Ba$, the above derivative is well defined almost everywhere.
Similarly, the $W^{1,1}(\gm)$-norm is defined by
\begin{equation}\label{w11-rel}
\|f\|_{W^{1,1}(\gm)}= \|f\|_{L^1(\gm)}  +  \left\| \frac{\d f}{\d \gm}\right\|_{L^1(\gm)}\,.
\end{equation}
Moreover,  the first and  second variations of $f$ with respect to the measure $\d \gm$ are defined respectively by 
\begin{equation}\label{TV}	
	 TV_\gm\left(f\right) = \sup \enscond{
		\int_{\Circ}f(s)\cdot \frac{\d g}{\d \gm(s)}(s)\, \d\gm(s) 
	}{
		g\in \mathrm{C}^{\infty}(\Circ,\RR^2),\|g\|_{L^\infty(\Circ,\R^2)}\leq 1
	} \,
\end{equation}
 and
\begin{equation}\label{TV2}
	 TV_\gm^2\left(f\right) = \sup \enscond{
		\int_{\Circ}f(s)\cdot \frac{\d^2 g}{\d \gm(s)^2}(s)\, \d\gm(s) 
	}{
		g\in \mathrm{C}^{\infty}(\Circ,\RR^2),\|g\|_{L^\infty(\Circ,\R^2)}\leq 1
	} \,.
\end{equation}
Finally, $BV(\gm)$ is the space of functions belonging to $L^1(\gm)$ with finite first variation $TV_\gm$. Analogously $BV^2(\gm)$ is the set of functions $W^{1,1}(\gm)$ with finite second variation $TV_\gm^2$.

\end{defn}

The next lemma points out some useful relationships between the quantities previously introduced.

\begin{lem}\label{equiv} For very $f \in BV^2(\gm)$ the following identities hold:

\begin{itemize}
\item[(i)] $TV_\gm^2\left(f\right)=TV_\gm\left(\frac{\d f}{\d \gm}\right)\,$;

\item[(ii)] $TV_\gm(f)=|Df|(\Circ)\,$;

\item[(iii)] $TV_\gm\left(f\right)=\left\| \frac{\d f}{\d \gm}\right\|_{L^1(\gm)}=\|f'\|_{L^1(\Circ,\RR^2)}\,$.

\end{itemize}

\end{lem}

\begin{proof} $(i)$ follows by integrating by parts. $(ii)$ follows from the definition of the derivative with respect to $\d \gm$ and \eqref{TV}. $(iii)$ follows from $(ii)$ and the definition of the derivative $\d/\d\gm$.
\end{proof}

Moreover, analogously  to Lemma 2.13 in \cite{Bruveris}, we have the following Poincar\'e inequality. The proof is similar to Lemma 2.13 in \cite{Bruveris}. 

\begin{lem} For every $f\in BV(\gm)$ it holds
\begin{equation}\label{embedding-gm}	
 \norm{f}_{L^\infty(\Circ,\R^2)} \leq  \frac{1}{\len(\gm)}\int_{\Circ} f\,\d \gm  + TV_\gm\left(  \frac{\d f}{\d \gm} \right)\,.
\end{equation}
\end{lem}

We can now define the norm on the tangent space to $\Ba$.

\begin{defn}[{\bf norm on the tangent space}] 

For every $\gm\in \Bb$, the tangent space at $\gm$ to $\Bb$, which is equal to $\BVd$ is endowed with the (equivalent) norm of the space
 $$BV^2(\gm)\,=\,BV^2(\Circ,\R^2;\d\gm)$$
introduced in Definition \ref{dgm}. 
More precisely, the $BV^2(\gm)$-norm is defined by
$$\|f\|_{BV^2(\gm)} = \int_{\Circ}\|f\|\|\gm'\|\,\d s + \int_{\Circ}\|f'\|\,\d s +  TV_\gm^2\left(f\right) \quad\forall\, f\in BV^2(\gm)\,.$$
Finally, we recall that 
\begin{equation}\label{change-arc}
\|f\|_{BV^2(\gm)}= \|f\circ\phi_\gm\|_{BV^2(\Circ, \R^2)}^{\gm}\,,
\end{equation}
where
$$
 \displaystyle{\|f\|_{BV^2(\Circ, \R^2)}^{\gm} = \len(\gm)\|f\|_{L^1(\Circ, \R^2)}+ \|f'\|_{L^1(\Circ, \R^2)}+\frac{1}{\len(\gm)}|D f|(\Circ) \quad \forall f\in BV^2(\Circ, \R^2)} \,.
 $$

\end{defn}

\begin{rem}[{\bf weighted norms}] Similarly to \cite{Bruveris}, we could consider some weighted $BV^2$-norms, defined as
$$\|f\|_{\BVd} = a_0\|f\|_{L^1(\Circ,\R^2)} +a_1\|f'\|_{L^1(\Circ,\R^2)} + a_2|D^2 f|(\Circ)\,,$$
where $a_i>0$ for $i=1,2,3$.
We can define the norm on the tangent space by the same constants.

One can easily satisfy that our results can be generalized to such a framework. In fact, this weighted norm is equivalent to the classical one and  the positive constants do not affect the bounds and the convergence properties that we prove in this work.
\end{rem}

The following proposition proves that $BV^2(\gm)$ and $\BVd$ represent the same space of functions with equivalent norms. 

\begin{prop}\label{equiv-norms}
Let $\gm \in \Ba$. 
The sets $BV^2(\gm)$ and $BV^2(\Circ,\R^2)$ coincide and their norms are equivalent.
More precisely, there exist two positive constants $M_\gm, m_\gm$  such that, for all $f \in BV^2(\Circ,\R^2)$
\eql{\label{eq-equi-norm}
	m_\gm \norm{f}_{BV^2(S^1,\RR^2)} \leq  \norm{f}_{BV^2(\gm)} \leq M_\gm\norm{f}_{BV^2(S^1,\RR^2)} \,.
}
\end{prop}

\begin{proof}
We suppose that $f$ is not equal to zero. For the $L^1$-norms of $f$, the result follows from~\eqref{cond-derivative-bis} and the  constants are given respectively by 
$$M_\gm^0 = \|\gm'\|_{L^\infty(\Circ,\RR^2)}\;, \quad m_\gm^0=\underset{s\in\Circ}{\mbox{essinf}}\,\|\gm'(s)\|.$$
Moreover, by Lemma \ref{equiv} $(iii)$, the $L^1(\gm)$ and $L^1(\Circ,\RR^2)$-norms of the respective first derivative coincide. So it is sufficient to obtain the result for the second variation  of $f \in BV^2(\gm)$.

By integration by parts, we have
$$	
	\int_{\Circ} f\cdot \frac{\d^2 g}{\d \gm^2}(s)\, \d\gm = 
	\int_{\Circ} \frac{f'}{\|\gm'\|}g'\, \d s  \,
$$
where we used the fact that $\frac{\d g}{\d \gm}=\frac{g'}{\|\gm'\|}$. 
This  implies in particular that 
\begin{equation}\label{Smoothness}TV^2_\gm(f) = \left|D \frac{f'}{\|\gm'\|}\right|(\Circ)\,.\end{equation}
Since $\frac{1}{\|\gm'\|} \in BV(\Circ,\RR^2)$ and $BV(\Circ,\RR^2)$ is 
a Banach algebra, we get 
$$
	\left|D \frac{f'}{\|\gm'\|}\right|(\Circ) \leq 
	|D f'|(\Circ) \left|D \frac{1}{\|\gm'\|} \right|(\Circ).
$$ 
Now, as $|Df'|(\Circ) =  |D^2f|(\Circ)$, applying the chain rule for $BV$-functions to $\left|D \frac{1}{\|\gm'\|}\right|(\Circ)$, we can set 
$$
	M_\gm^2=   \|\gm'\|_{BV(\Circ,\RR^2)}/\underset{s\in\Circ}{\mbox{essinf}}\,\|\gm'(s)\|^2\,.
$$
On the other hand, we have
$$
	\int_{\Circ} f'g'\, \d s  = 
	\int_{\Circ}\frac{\d f}{\d \gm} \frac{\d g}{\d \gm}\,\|\gm'\| \d\gm   \,
$$
so that
$$
	|D f'|(\Circ) = TV_\gm\left(\frac{\d f}{\d \gm} \|\gm'\| \right)\,,
$$
and, because of Lemma \ref{equiv}$(i)$, we get
\begin{equation}\label{m_GA}
	|D^2 f|(\Circ)\leq TV_\gm(\|\gm'\|)TV^2_\gm(f) \,.
\end{equation}
Therefore, by the chain rule for $BV$-functions, the result is proved by taking the constant
$$
m_\gm^2= \frac{1}{\|\gm'\|_{BV(\Circ,\RR^2)}} \, \cdot
$$
The lemma ensues setting
\begin{equation}\label{equiv-down}
\begin{array}{lll}
	M_\gm & = &\max \,\{M_\gm^0,M_\gm^2\}   
	=  \max \,\left\{\|\gm'\|_{L^\infty(\Circ,\RR^2)}\,,\;  
	\|\gm'\|_{BV(\Circ,\RR^2)}/\underset{s\in\Circ}{\mbox{essinf}}\,\|\gm'(s)\|^2\right\}\,,\\
 m_\gm & = & \min \,\{m_\gm^0,m_\gm^2\}  =  \min \,\left\{\underset{s\in\Circ}{\mbox{essinf}}\,\|\gm'(s)\|\,,\;1/\|\gm'\|_{BV(\Circ,\RR^2)}\right\}\,.
 \end{array}
\end{equation}

\end{proof}

\subsection{Paths Between $BV^2$-Immersed Curves and Existence of Geodesics}
\label{existence_geod}

In this section, we define the set of admissible paths between two $BV^2$-immersed curves and  a $BV^2$ Finsler metric on $\Bb$. In particular we prove that a minimizing geodesics for the defined Finsler metric exists for any given couple of curves.

\begin{defn}[{\bf paths in $\Ba$}]
For every $\gm_0,\gm_1\in \Ba$, we define a path in $\Ba$ joining  $\gm_0$ and $\gm_1$ as a function 
$$
	\GA: t\in [0,1] \mapsto \GA(t) \in \Ba \quad \forall t\in [0,1]
$$
such that
\begin{equation}\label{initial_conditions} 
 	\GA(0) = \gm_0\quad \GA(1) = \gm_1\,.
\end{equation}
For every $\gm_0,\gm_1\in \Ba$, we denote $\mathcal{P}(\gm_0,\gm_1)$ the class of all paths joining $\gm_0$ and $\gm_1$, belonging to $H^1([0,1],\BVd)$, and such that $\GA(t)\in\Bb$ for every $t\in [0,1]$. 

We recall that $H^1([0,1],\BVd)$ represents the set of $\BVd$-valued functions whose derivative belongs to $L^2([0,1],\BVd)$. We refer to \cite{Iounesco} for more details about Bochner space of Banach-valued functions. It holds in particular  
\begin{equation}\label{flow}
	\forall \, s\in \Circ, \quad
	\int_0^1 \GA_t(t)(s) \d t = \gm_1(s)-\gm_0(s)\,,
\end{equation}
where $\GA_t$ denotes the derivative of $\GA$ with respect to $t$. In the following $\GA'(t)$ denotes the derivative of the curve $\GA(t)\in \BVd$ with respect to $s$. Finally, for every $t$ and for every $s$, it holds
\begin{equation}\label{t-curve-der}
\GA(t)(s) = \int_0^t \GA_\tau(\tau)(s) \d \tau +\gm_0(s)
\,,\quad \quad
\GA'(t)(s) = \int_0^t \GA_\tau'(\tau)(s) \d \tau +\gm_0'(s)\,.
\end{equation}
\end{defn}

\begin{defn}[{\bf geodesic paths in $\Ba$}]\label{defn-geodesic-paths}
 For every path $\GA$ we consider the following energy 
\begin{equation}\label{energy-bv2}
	E(\GA)=\int_0^1 \|\GA_t(t)\|_{BV^2(\GA(t))}^2\,\d t.
\end{equation}
The geodesic distance between $\gm_0$ and $ \gm_1$ is denoted by $d(\gm_0,\gm_1)$ and defined by 
\begin{equation}\label{problem}
	d^2(\gm_0,\gm_1) = {\rm inf}\enscond{ E(\GA) }{ 
		\GA\in \mathcal{P}(\gm_0,\gm_1) 
	}\,.
\end{equation}
A geodesic between $\gm_0$ and $\gm_1$ is a path $\tilde{\GA}\in \mathcal{P}(\gm_0,\gm_1)$ such that
$$	
	E(\tilde{\GA}) = d^2(\gm_0,\gm_1).
$$
\end{defn}

Note that because of the lack of smoothness of the $BV^2$-norm over the tangent space, it is not possible to define an exponential map. Geodesics should thus be understood as paths of minimal length. Recall that the existence of (minimizing) geodesics is not guaranteed in infinite dimensions. 

\begin{rem}[{\bf time reparameterization and geodesic energy}]\label{time-param} 
We point out that, as in  Remark~\ref{arc-len}, we can reparameterize  every non-trivial  homotopy $\GA$ (i.e. satisfying $E(\GA) \neq 0$) with respect to the time-constant speed parameterization, defined as the inverse of the following parameter:
$$t_\GA : [0,1] \rightarrow [0,1]$$
$$\foralls t\in [0,1], \quad 
	t_\GA(t) =\frac{1}{E_1(\GA)} \int_{0}^t \,\|\GA_\tau(\tau)\|_{BV^2(\GA(\tau))} \, \d\tau\,,$$
	where 
	$$E_1(\GA)=\int_{0}^1 \,\|\GA_\tau(\tau)\|_{BV^2(\GA(\tau))} \, \d\tau\,. $$
In the following we show the link between the $L^1$ and $L^2$ geodesic energies via a time reparameterization.

Note that, we can suppose that there is no interval $I\subset [0,1]$ such that 	$\|\GA_t(t)\|_{BV^2(\GA(t))}= 0$ a.e. on $I$. Otherwise,  we can always consider, by a  reparameterization,  the homotopy $\tilde{\GA}$ such that $\tilde{\GA}([0,1]) =\GA_{|[0,1]\setminus I}([0,1])$, which  is such that  $E(\GA)=E(\tilde{\GA})$ and $t_{\tilde{\GA}}$ is strictly increasing. 

Then, up to such a reparameterization, we assume that  $t_\GA$ is a strictly monotone (increasing) continuous function from $[0,1]$ onto $[0,1]$ so it is invertible and  we can define the time-constant speed parameterization $t_\GA^{-1}$. Now, if the homotopy is parameterized  with respect to such a  parameter then it satisfies
\begin{equation}\label{time-vel}
	\|\GA_t(t)\|_{BV^2(\GA(t))} = E_1(\GA)\quad \foralls t\in [0,1]\,.
\end{equation} 
In particular, for a generic homotopy $\GA$,  we have 
$$(E_1(\GA))^2= E(\GA\circ t_{\GA}^{-1})\,.$$
This implies  that the minimizers of $E$ satisfy \eqref{time-vel} and coincide with the time-constant speed reparameterizations of the  minimizers of $E_1$.
 This  justifies  the definition of the geodesic energy by $E$ instead of $E_1$ that formally represents the length of the path.
We refer to  \cite[Theorem 8.18  and Corollary 8.19, p.175]{Younes-book}  for more details.
\end{rem}

We prove now that the constants  $m_{\GA(t)}$ and $M_{\GA(t)}$ defined in \eqref{equiv-down} are uniformly bounded on minimizing paths.
To this end we need the following lemma.

\begin{lem}\label{no-collapse}
Let $\GA\in \mathcal{P}(\gm_0,\gm_1)$. Then the following properties hold:
\begin{enumerate}
\item The function
$$t\mapsto g(t)= \|\GA'(t)\|_{L^\infty(\Circ,\RR^2)}$$
belongs to $C([0,1], \R)$, so, in particular, it admits a  maximum and a positive minimum on $[0,1]$. Similarly,  the  functions $t\mapsto \len(\GA(t))$ and $t\mapsto \|\GA'(t)(s)\|$ (for a.e. $s\in \Circ$) are also  continuous. 
\item For every $t\in [0,1]$ we have
\begin{equation}\label{length}
\len(\gm_0)e^{ -E(\GA)} \leq \len(\GA(t)) \leq \len(\gm_0)e^{E(\GA)}\,,
\end{equation}
and for a.e. $s\in \Circ$ and  for every $t\in [0,1]$, we have \begin{equation}\label{length2}	(\underset{s\in \Circ}{\rm{essinf}}\,\|\gm_0'(s)\|)e^{-  E(\GA)}	\leq \|\GA'(t)(s)\|\leq 	\|\gm_0'\|_{\leb{\infty}}e^{  E(\GA)}\,.\end{equation}

\end{enumerate}
\end{lem}

\begin{proof}
1. By Definition~\ref{defn-geodesic-paths}, every $\GA\in \mathcal{P}(\gm_0,\gm_1)$ belongs  to $H^1([0,1],\BVd)$ so, in particular, to $C([0,1],\BVd)$. Now, as $\BVd$ is embedded in $L^\infty(\Circ,\RR^2)$, we get the continuity of $g$. By a similar argument we get the continuity of the   functions $t\mapsto \len(\GA(t))$ and $t\mapsto \|\GA'(t)(s)\|$ (for a.e. $s\in \Circ$).


\medskip
2. We recall that $\GA(t)\in \Bb$ for every $t$ so that $\GA'(t)$ satisfies \eqref{cond-derivative} for every $t$. In particular the derivative $\GA'(t)$ is well defined a.e. on $\Circ$. By Remark~\ref{time-param} we can suppose that the time velocity satisfies \eqref{time-vel}. We have 
$$\frac{\partial \len(\GA(t))}{\partial t} = 
\int_{\Circ} \left\langle\frac{\GA_t'(t)}{\|\GA'(t)\|}, \GA'(t)\right\rangle\, \d s
\leq\left\|\frac{\d \GA_t(t)}{\d \GA(t)}\right\|_{L^\infty(\GA(t))}\len(\GA(t))\,,$$
and, as $\frac{\d \GA_t(t)}{\d \GA(t)}$ has null average, by \eqref{embedding-gm}, we have 
$$\left\|\frac{\d \GA_t(t)}{\d \GA(t)}\right\|_{L^\infty(\Circ,\RR^2)}\leq E(\GA)\,.$$
This implies
$$ \frac{\partial \log( \len(\GA(t)))}{\partial t}\leq E(\GA) \quad \forall \, t\in [0,1],$$
and, by integrating between $0$ and $t$, we get 
$$\len(\gm_0)e^{ -E(\GA)} \leq \len(\GA(t)) \leq \len(\gm_0)e^{ E(\GA)}\,. $$ 
For the second inequality we remark that, because of \eqref{t-curve-der}, $\|\GA'(t)(s)\|$ is differentiable with respect to $t$ for a.e. $s$. Then   $$\frac{\partial \|\GA'(t)(s)\|}{\partial t} = \left\langle\frac{\GA_t'(t)(s)}{\|\GA'(t)(s)\|}, \GA'(t)(s) \right\rangle \leq\left\|\frac{\d \GA_t(t)(s)}{\d \GA(t)}\right\|_{L^\infty(\Circ,\RR^2)}\|\GA'(t)(s)\|\,$$and, as above, we get  $$ \frac{\partial \log( \|\GA'(t)(s)\|)}{\partial t}\leq E(\GA) \quad 	\forall\,\mbox{a.e.}\; s\in \Circ\,\quad  \forall \, t\in [0,1].$$The result follows by integrating with respect to $t$. 
\end{proof}

\begin{prop}\label{global_bound_Mm} 
Let $\gm_0,\gm_1\in \Ba$. Then, for every  $\GA\in \mathcal{P}(\gm_0,\gm)$, there exist  two positive constants $C_1,C_2$  depending on $\gm_0$  (see \eqref{C1C2}) such that
\begin{equation}\label{constantsMm}
C_1e^{-2E(\GA)}\leq m_{\GA(t)}\leq  M_{\GA(t)}\leq C_2e^{2E(\GA)} \quad \forall \, t \in [0,1]\,,
\end{equation}
where  the constants $m_{\GA(t)}$ and $M_{\GA(t)}$ are defined in \eqref{equiv-down}.

\end{prop}

\begin{proof}
 
Up to a time reparameterization, we can suppose that the homotopy satisfies \eqref{time-vel}. Because of \eqref{length2}, we have 
\begin{equation}\label{global_length}
\begin{array}{ll}
	(\underset{s\in \Circ}{\rm{essinf}}\,\|\gm_0'(s)\|)e^{-E(\GA)} 
	\leq \underset{s\in\Circ}{\rm{essinf}}\,\|\GA'(t)(s)\| \,,\vspace{0.2cm}\\
	 \|\GA'(t)\|_{L^\infty(\Circ,\RR^2)} \leq 
	\|\gm_0'\|_{\leb{\infty}} e^{ E(\GA)}
\end{array}
\end{equation}
for every $t\in [0,1]$. By setting  $f(t) =\GA_t(t)$ in \eqref{m_GA}, we get 

\begin{equation}\label{m_GA_2}
	|D^2 \GA_t(t)|(\Circ)\leq TV_{\GA(t)}(\|\GA'(t)\|)E(\GA)\,.
\end{equation}

Thus
$$
\|\GA'(t)-\gm_0'\|_{BV(\Circ,\RR)}\, \leq  \,\int_0^t \,\|\GA'_{\tau}(\tau)\|_{BV(\Circ,\RR)} \d \tau \,=\, \int_0^t[\|\GA'_{\tau}(\tau)\|_{\leb{1}} + |D^2 \GA_{\tau}(\tau)|(\Circ) ] \d \tau $$
and, by \eqref{w11-rel}, 
\eqref{m_GA_2}, and \eqref{time-vel}, we have 
$$\|\GA'(t)-\gm_0'\|_{BV(\Circ,\RR)} \,\leq \, E(\GA) +   \int_0^t TV_{\GA(t)}(\|\GA'(t)\|)E(\GA) \d \tau \,.
$$
In particular, by the chain rule for $BV$-functions, we have
$$ TV_{\GA(t)}(\|\GA'(t)\|)\leq\, \|\GA'(t)\|_{BV(\Circ,\RR)}\,.$$

Then
$$\|\GA'(t)\|_{BV(\Circ,\RR)}\leq  \|\gm_0'\|_{BV(\Circ,\RR)}+E(\GA) +   \int_0^t \|\GA'(t)\|_{BV(\Circ,\RR)}E(\GA) \d \tau$$
and, by  Gronwall's inequality, we get 
\begin{equation}\label{bound-BV}
\|\GA'(t)\|_{BV(\Circ,\RR)}\leq (\|\gm_0'\|_{BV(\Circ,\RR)}+E(\GA))e^{E(\GA)}\,.
\end{equation} 
From \eqref{global_length} and \eqref{bound-BV}, it follows  that 
\begin{equation}\label{global_length_2}
\begin{array}{ll}
\displaystyle{\frac{e^{-E(\GA) }}{ (\|\gm_0'\|_{BV(\Circ,\RR)}+E(\GA))}	\leq \frac{1}{\|\GA'(t)\|_{BV(\Circ,\RR^2)}}}\,, \vspace{0.2cm}\\
\displaystyle{\frac{\|\GA'(t)\|_{BV(\Circ,\RR^2)}}{\underset{s\in\Circ}{\mbox{essinf}}\,\|\GA'(t)(s)\|^2} \leq 	\frac{(\|\gm_0'\|_{BV(\Circ,\RR)}+E(\GA))}{\underset{s\in \Circ}{\rm{essinf}}\,\|\gm_0'(s)\|^2}}e^{2E(\GA)}\,.
\end{array}
\end{equation}
The result follows from \eqref{global_length} and \eqref{global_length_2} by setting
\begin{equation}\label{C1C2}
\begin{array}{ll}
C_1= &\displaystyle{ \min\, \left\{ \underset{s\in \Circ}{\rm{essinf}}\,\|\gm_0'(s)\|\,,\,\frac{1}{ (\|\gm_0'\|_{BV(\Circ,\RR)}+E(\GA))}\right\} }\,,\vspace{0.2cm}\\

C_2=& \displaystyle{
\max \, 
\left\{ \|\gm_0'\|_{\leb{\infty}}\, ,\, \frac{(\|\gm_0'\|_{BV(\Circ,\RR)}+E(\GA))}{\underset{s\in \Circ}{\rm{essinf}}\,\|\gm_0'(s)\|^2} 
\right\}}\,.

\end{array}
\end{equation}
\end{proof}

\begin{rem}[{\bf weak topologies in Bochner spaces}]\label{topology-bochner}
Proposition \eqref{equiv-norms} proves the local equivalence between the Finsler metric $BV^2(\gm)$ and the ambient metric $BV^2([0,1], \RR^2)$. This implies in particular that every minimizing sequence $\{\Gamma^h\}$ of $E$ is bounded in $H^1([0,1],BV^2([0,1],\RR^2))$. In fact, because of \eqref{constantsMm}, we have
$$
\int_0^1 \|\GA_t^h(t)\|^2_{BV^2([0,1],\RR^2)}\,\d t\,\leq \,\int_0^1 \frac{\|\GA_t^h(t)\|^2_{BV^2(\GA^h(t))}}{m_{\GA^h(t)}^2} \,\d t \,\leq \, \frac{1}{m^2}E(\GA)\, \quad m=\underset{h}{\inf}\,\underset{t\in [0,1]}{\rm{essinf}}\,m_{\GA^h(t)}\,.
$$
Then, we could  use some compactness results for the Bochner space $H^1([0,1],BV^2([0,1],\RR^2))$ with respect to some weak topology. 
Now, to our knowledge, the usual  weak and weak* topologies on the Bochner space $H^1([0,1],BV^2([0,1],\RR^2))$ can not be suitably characterized, so that, working with these topologies, prevents us from describing the behavior of the minimizing sequence. 
For instance, the question of the convergence of the curves $\{\Gamma^h(t)\}$ at time $t\in [0,1]$ cannot be answered if we do not have a precise characterization of the topology used to get  compactness. 

We recall in this remark the main issues linked to the characterization of the weak topologies for the Bochner space of $BV^2$-valued functions.

Firstly, we recall that, for every Banach space $B$, the dual space of the Bochner space $H^1([0,1],B)$  is represented by $H^1([0,1],B')$ if and only if the dual space $B'$ has the Radon-Nicodym property (RNP) \cite{Bochner_dual, Bochner_dual_2}. 

This means that, for every measure $\mu: \mathcal{M}([0,1])\rightarrow B$ which has bounded variation and is absolutely continuous with respect to $\lambda$ ($\mathcal{M}([0,1])$ denotes the class of Lebesgue-measurable sets of $[0,1]$ and $\lambda$ the one dimensional Lebesgue measure),  there exists a (unique) function $f\in L^1([0,1], B)$ such that 
\begin{equation}\label{rnp}\mu(A)=\int_A f(t) \,\d \lambda(t) \quad \quad \forall\, A\in \mathcal{M}([0,1])\,.
\end{equation}
This essentially means that the Radon-Nicodym theorem holds for $B$-valued measures. More precisely the Radon-Nicodym derivative $\mu/\lambda$ is represented by a $B$-valued function.
Spaces having the RNP are, for instance,  separable dual spaces and reflexive spaces, so, in particular, Hilbert spaces. However, $L^1(K)$, $L^\infty(K)$, and $C(K)$, where $K$ is a compact set of $\RR$, do not have the RNP.

Now, in order to apply to our case such a result, we should be able to completely characterize the dual of $BV^2([0,1],\RR)$, which represents at the moment an open problem \cite{MZ, Pauw}. Therefore we cannot characterize the weak topology of our initial space.

Another possibility to apply the previous duality result is to consider $BV^2([0,1],\R)$ as the dual  of a Banach space $B$. Then, by proving that it has the RNP and applying the duality result, we could  characterize the weak* topology of $BV^2([0,1],\R)$.  

In fact, according to the characterization of the dual of  Bochner spaces cited above, we could write $H^1([0,1],BV^2([0,1],\R))$ as the dual of $H^1([0,1],B)$.
Unfortunately, $BV^2([0,1],\R)$ does not have the RNP as  is shown  by the following example. 

We consider the following $BV^2([0,1],\RR)$-valued measure
$$\mu(A)\,=\,(\,x\mapsto \varphi_A(x)=\lambda(A\cap (0,x))\,)\quad \quad\forall\, A\in \mathcal{M}([0,1])\,.$$
We can easily satisfy that, for every $A$, $\varphi_A\in BV^2([0,1],\R)$ and $\varphi_A'=\mathbbm{1}_A$. Moreover, if $A$ is Lebesgue negligible we have $|\mu|(A)=0$, which means that $|\mu|$ is absolutely continuous with respect to the Lebesgue measure $\lambda$. 

However, if there exists a function $f\in L^1([0,1],BV^2([0,1],\RR))$ satisfying \eqref{rnp}, then,   we should have in particular 
$$\lambda(A\cap (0,x))\,=\,\int_A f(t)(x)\, \d \lambda(t)\,,$$
where, for every $t$,  $f(t)(x)$ denotes the value of $f(t)$ at $x$. Then, for every $B\subset [0,1]$, we obtain 
$$\int_B\lambda(A\cap (0,x))\,\d \lambda(x)\,=\,\int_B\int_A f(t)(x)\,d \lambda(t)\d \lambda(x)\,$$
which implies that $f(t)(x)=1$ a.e. if $t<x$ and $f(t)(x)=0$ a.e. if $t\geq x$. This is of course in contradiction with the fact   $f(t)$ has to be continuous because it is a $BV^2$-function.
Previous examples and considerations show that, in our case, the weak and weak* topologies are not suitably characterized in order to give meaningful information on the limit.

In order to prove the existence of a geodesic we use a new proof strategy, which is  inspired to the technique proposed in \cite{MN} and is detailed in the proof of the next theorem. We also point out that 
this actually defines  a suitable topology, which allows  us to get semicontinuity and compactness in our framework (see Definition \ref{sigma}).
\end{rem}

We can now prove an existence result for geodesics. 

\begin{thm}[{\bf existence of geodesics}]\label{local_existence}
Let $\gm_0, \gm_1 \in \Ba$ such that $d(\gm_0,\gm_1)< \infty$.
Then, there exists a geodesic between $\gm_0$ and $\gm_1$.
\end{thm}

\begin{proof}
Let $\{\GA^h\} \subset \mathcal{P}(\gm_0,\gm_1)$ be a minimizing sequence for $E$ so that $E(\GA^h)\rightarrow \inf\,E$. Without loss of generality we can suppose $\sup_h \; E(\GA^h)< +\infty$.  We also remark that, from the previous lemma, it follows that
$$0 \,<\, \underset{h}{\inf}\,\underset{t\in [0,1]}{\mbox{essinf}}\, m_{\GA^h(t)}\, <\,\underset{h}{\sup}\,\underset{t\in [0,1]}{\mbox{esssup}}\, M_{\GA^h(t)}\, <\,+\infty\,.
$$
Moreover, we can suppose (up to a time reparameterization) that every homotopy is parameterized with respect to the time-constant speed parameterization. Then, we can assume that  $\{\GA^h_t(t)\}$  satisfies \eqref{time-vel} for every $h$. 

\medskip
\noindent{\em Step 1: Definition of  a limit path.} For every $n>1$ we consider the dyadic decomposition of $[0,1]$ given by the intervals  
\eql{\label{eq-dyadic}
	I_{n,k} = [\frac{k}{2^n},\frac{k+1}{2^n}[
	\qforq k \in [0,2^n-1]
}
and, for every $t\in [0,1]$  we define 
$$f_{n}^h(t)=2^n \int_{I_{n,k}} \GA^h_\tau(\tau)\, \d \tau, $$
where $I_{n,k}$ is the interval containing $t$. Remark that, for every $n$ and $h$, $f_{n}^h: [0,1]\rightarrow BV^2(\Circ,\RR^2)$  is  piecewise constant with respect to the family $\{I_{n,k}\}$, and
\begin{equation}\label{dom-conv}
\int_0^1 f_n^h(t)\, \d t = \int_0^1 \GA^h_t(t)\, \d t = \GA(1)-\GA(0)\,.
\end{equation}  
Now, setting
$m = \underset{h}{\inf}\,\underset{t\in [0,1]}{\mbox{essinf}}\, m_{\GA^h(t)}\, >\,0
$, by Jensen's inequality and \eqref{eq-equi-norm}, we get 
\begin{equation}\label{bound0}
\|f^h_n(t)\|^2_{BV^2(\Circ,\RR^2)}\leq 2^n\int_{I_{n,k}} \frac{\|\GA_t^h(t)\|^2_{BV^2(\GA(t))}}{m_{\GA^h(t)}^2}\,\d t \,\leq \, \frac{2^n}{m^2}E(\GA^h)\quad \foralls \, t\in [0,1]\,.
 \end{equation}
So, by $2^n$ successive extractions, we can take
a subsequence (not relabeled) and  a  piecewise constant (with respect to the family $\{I_{n,k}\}$) function $f_{n}^\infty: [0,1]\rightarrow BV^2(\Circ,\RR^2)$  such that 
$$\foralls n, \quad f^h_n(t) \overset{*-BV^2}{\rightharpoonup} f^\infty_n(t) \quad \foralls t\in [0,1]\,,$$
and
\begin{equation}\label{conv} \int_0^1\|f_n^\infty(t)\|^2_{BV^2(\Circ,\RR^2)} \d t \leq \underset{h\rightarrow \infty}{\liminf} \;\; \int_0^1\|f^h_n(t)\|^2_{BV^2(\Circ,\RR^2)}\d t\,.
\end{equation}
Moreover we can write 
$I_{n,k}=[k2^{-n}, (2k+1)2^{-n-1}[\cup [(2k+1)2^{-n-1}, (k+1)2^{-n}[$
and 
$$f^h_{n+1}(t) = 2^{n+1}\int_{I_{n+1,2k}} \GA^h_t(t) \d t \quad \foralls t\in [k2^{-n}, (2k+1)2^{-n-1}[\,,$$
$$f^h_{n+1}(t) = 2^{n+1}\int_{I_{n+1,2k+1}}\GA^h_t(t) \d t  \quad \foralls t\in [(2k+1)2^{-n-1}, (k+1)2^{-n}[\,,$$
and therefore
$$\int_{I_{n,k}}f^h_{n+1}(t) \d t =  \int_{I_{n,k}}\GA^h_t(t) \d t\,,$$
$$f^h_n(t) = 2^n \int_{I_{n,k}}f^h_{n+1}(t)\d t \quad \foralls t\in I_{n,k}.$$
Then, by the dominated convergence theorem, we get
\begin{equation}\label{marti}
 f_n^\infty(t) = 2^n \int_{I_{n,k}}f_{n+1}^\infty(t) \d t\,,
\end{equation}
which implies that $\{f^\infty_n\}$ is a $BV^2(\Circ,\RR^2)$-valued martingale \cite{Durett, Chatterji}. We note that $\{f^\infty_n\}$ is a martingale with respect to the probability space $[0,1]$ equipped with Lebesgue measure and that the filtration is defined by the increasing sequence of $\sigma$-algebras generated (at every time $n$) by the intervals $\{I_{n,0},...,I_{n,2^n-1}\}$.

Moreover, by \eqref{conv}, Fatou's lemma, and,~\eqref{bound0},  we get
$$
\begin{array}{ll}
\displaystyle{ \int_0^1\|f_n^\infty(t)\|^2_{BV^2(\Circ,\RR^2)} \d t }&\displaystyle{\leq \underset{h\rightarrow \infty}{\liminf} \;\; \int_0^1\|f^h_n(t)\|^2_{BV^2(\Circ,\RR^2)}\d t}\\
 & \displaystyle{\leq \underset{h\rightarrow \infty}{\liminf} \;\; \frac{E(\GA^h)}{m^2}}\,.
 \end{array}
$$
\par Now, as $BV^2(\Circ,\RR^2)$ is embedded in $H^1(\Circ,\RR^2)$, this implies that $\{f^\infty_n\}$ is a bounded  martingale in $L^2([0,1],H^1(\Circ,\RR^2))$ so, by the convergence theorem for martingales \cite[Theorem 4]{Chatterji},  $f^\infty_n(t)\rightarrow f(t)$ in $H^1(\Circ,\RR^2)$ for almost every $t$. 
Note also that, as $f^\infty_n\in BV^2(\Circ,\RR^2)$ and the second variation is lower semicontinuous with respect to the $W^{1,1}(\Circ,\RR^2)$-convergence, we actually get $f\in L^2([0,1],BV^2(\Circ,\RR^2))$.

We can now define a candidate to be a minimum  of $E$ by setting
\begin{equation}\label{formula}
\GA^\infty(t)= \int_0^t f(\tau)\, \d \tau +\GA(0) \quad \foralls \; t\in [0,1]\,.
\end{equation}

\medskip
\noindent
{\em Step 2: $\GA^\infty$ is a geodesic path.}
\, 
We can easily satisfy that  $\GA^\infty\in H^1([0,1],BV^2(\Circ,\RR^2))$ and that $\GA^\infty$ satisfies~\eqref{flow}. In fact, by the dominated convergence theorem and~\eqref{dom-conv}, we have
$$\int_0^1 f(\tau)\, \d \tau = \underset{n\rightarrow \infty}{\lim} \, \int_0^1 f_n^\infty(\tau)\, \d \tau = \underset{n\rightarrow \infty}{\lim}\underset{h\rightarrow \infty}{\lim}\int_0^1 f_n^h(\tau)\, \d \tau = \GA(1)-\GA(0)\,.$$
This implies in particular that $\GA^\infty$ satisfies~\eqref{initial_conditions}. In order to prove that  $\GA^\infty\in \mathcal{P}(\gm_0,\gm_1)$ we have to show that $\GA^\infty(t)\in \Bb$  for every $t$. 

Below, we prove that $ \GA^h(t) \rightarrow \GA^\infty(t)$ in ${W^{1,1}(\Circ,\RR^2)}$ for every $t$. This implies (up to a subsequence) the a.e. convergence and, because of \eqref{length2}, we get
$$(\underset{s\in \Circ}{\rm{essinf}}\,\|\gm_0'(s)\|)e^{-  \inf E}	\leq \|(\GA^\infty)'(t)(s)\|$$
for every $t$ and for a.e.-$s$. This proves in particular that $\GA^\infty(t)$ satisfies \eqref{cond-derivative-bis} for every $t$ so that   $\GA^\infty\in \mathcal{P}(\gm_0,\gm_1)$.

We denote by $\GA_n^\infty$ and $\GA_n^h$ the paths defined by $f_n^\infty$ and $f_n^h$  through~\eqref{formula}, respectively. 
Now, as  $\{\GA^h_t(t)\}$  satisfies \eqref{time-vel} for every $h$ the norms  $\|\GA^h_t(t)\|_{L^\infty(\Circ,\RR^2)}$ are uniformly bounded. Then, by the definition of $f^h_n$ and a straightforward computation, we get that 
$\| \GA^h(t) -\GA^h_n(t)\|_{W^{1,1}(\Circ, \RR^2)}$
is small for $n$ large enough.  

 Moreover, as  $f^h_n(t) \overset{*-BV^2}{\rightharpoonup} f^\infty_n(t)$ for every $t$, from the dominated convergence theorem, it follows that $\| \GA^\infty_n(t) -\GA^h_n(t)\|_{{W^{1,1}(\Circ,\RR^2)}}\rightarrow 0$ for every $t$ as  $h\rightarrow \infty$ for every $n$.
Similarly, as  $f^\infty_n\rightarrow f$ in $H^1(\Circ,\RR^2)$, 
$\| \GA^\infty_n(t) -\GA^\infty(t)\|_{W^{1,1}(\Circ,\RR^2)}$ is small for $n$ large enough.

Finally, this implies that 
\begin{equation}\label{cons-length}
\|\GA^h(t) -\GA^\infty(t)\|_{{W^{1,1}(\Circ,\RR^2)}}\rightarrow 0\,\quad \mbox{as}\; h\rightarrow 0\,\quad \forall\,t\in [0,1]\,.
\end{equation}
By the same arguments we can show that
\begin{equation}\label{conv-v}
\| \GA^h_t(t) -\GA^\infty_t(t)\|_{W^{1,1}(\Circ,\RR^2)}\rightarrow 0\,\quad \mbox{as}\; h\rightarrow 0\,\quad \forall\,t\in [0,1]\,.
\end{equation}
We prove now that $\GA^\infty$ is a minimizer of $E$.  
We recall that we have supposed (up to a time reparameterization) that   $\{\GA^h_t(t)\}$ is bounded in $BV^2(\Circ,\RR^2)$ for every $t$. 
\par Because of \eqref{length}, \eqref{length2}, and, Lemma \ref{conv-w11}, as   $\GA^h(t)$ converges in $W^{1,1}(\Circ,\RR^2)$ towards $\GA^\infty(t)$,  the constant speed parameterizations at time $t$, denoted respectively by  $\phi_{\GA^h(t)}$ and $\phi_{\GA^\infty(t)}$, also converge in $W^{1,1}(\Circ, \Circ)$. Moreover, according to \eqref{change-arc}, for a fixed time $t$, we have
$$ \|\GA^{h}_t(t)\|_{BV_2(\GA^h(t))} =  \|\GA^{h}_t \circ \phi_{\GA^h(t)}\|^{\GA^h(t)}_{BV^2(\Circ,\RR^2)}\,.$$
Since $\{\GA^h_t(t)\}$ is bounded in $BV^2(\Circ,\RR^2)$ and  the two terms involved in the composition  converge in $W^{1,1}(\Circ,\RR^2)$, we have
$$\|\GA^{h}_t \circ \phi_{\GA^h(t)} - \GA^{h}_t \circ \phi_{\GA^\infty(t)}\|_{W^{1,1}(\Circ,\RR^2)}\leq \|\GA^h_t\|_{BV^2(\Circ,\RR^2)}\| \phi_{\GA^h(t)} - \phi_{\GA^\infty(t)}\|_{W^{1,1}(\Circ,\RR^2)}\,,$$
$$\|\GA^{h}_t \circ \phi_{\GA^\infty(t)} -	 \GA^{\infty}_t \circ \phi_{\GA^\infty(t)}\|_{W^{1,1}(\Circ,\RR^2)}\leq C\|\GA^{h}_t  -	 \GA^{\infty}_t\|_{W^{1,1}(\Circ,\RR^2)}\,,$$
where the constant 
$$C=\max\,\{1, \|(\varphi_{\GA^\infty(t)}^{-1})'\|_{L^\infty(\Circ,\RR^2)}\}\,$$
is bounded because of Lemma \ref{bound-rep-bv}.
This implies in particular that 
\begin{equation}\label{sci2}
\GA^{h}_t \circ \phi^h(t)  \overset{W^{1,1}(\Circ,\RR^2)}{\rightarrow} \GA^{\infty}_t \circ \phi^\infty(t)\,.
\end{equation}
Now, we note that, because of the $W^{1,1}$-convergence, we have $\len(\GA^h(t))\rightarrow \len(\GA(t))$. Moreover,  the second variation is lower semicontinuous with respect to the  $W^{1,1}$-convergence. 

Then, by \eqref{change-arc}, for every  $t$ we get

\begin{multline}\label{sci}
	\| \GA^{\infty}_t(t)\|_{BV^2(\GA^\infty(t))}
	=  \| \GA^\infty_t \circ \phi_{\GA^\infty(t)} \|^{\GA^{\infty}(t)}_{BV^2(\Circ,\RR^2)} \vspace{0.3cm}\\
	\leq \underset{h\rightarrow \infty}{\liminf} \, \| \GA^h_t \circ \phi_{\GA^h (t)}\|^{\GA^h(t)}_{BV^2(\Circ,\RR^2)}
	= \underset{h\rightarrow \infty}{\liminf} \,\|\GA^{h}_t(t)\|_{BV^2(\GA^h(t))}\,.
\end{multline}

By integrating the previous inequality and using Fatou's lemma we get that $\GA^\infty$ minimizes $E$, which ends the proof.
\end{proof}

\begin{rem} In order to get the semicontinuity's inequality~\eqref{sci}, we actually just would need the convergence in~\eqref{sci2} with respect to the $BV^2$-weak topology. We recall that we do not know a characterization of the dual space of $BV$, which explains the choice of the weak-* topology (i.e. we look at $\BVd$ as a dual Banach space) in the previous proof. 

Moreover,  the martingale approach allows one to get strong convergence in  $W^{1,1}$ without applying any strong-compactness criterion for Sobolev spaces. This is a key point of the proof because the $BV^2$-norm  is semicontinuous with respect to the strong $W^{1,1}$-topology. 
\end{rem}

Inspired by the previous proof we can define the following topology on $H^1([0,1],\Bb)$:

\begin{defn}[{\bf $\sigma$-topology}]\label{sigma} Let $\{\GA^h\}\subset H^1([0,1],\BVd)$ and $\GA\in H^1([0,1],\BVd)$. Let $\{I_{n,k}\}_{n,k}$ be the collection of the intervals giving the dyadic decomposition of $[0,1]$ defined in~\eqref{eq-dyadic}. We say that $\GA_h$ converges to $\GA$ with respect to the $\sigma$-topology (denoted by $\GA^h\overset{\sigma}{\rightarrow}\GA$) if, for every $n$,  there exists a sequence of piecewise constant functions $\{f_n^\infty\}$ on  $\{I_{n,k}\}_{n,k}$  such that the following hold.
\begin{itemize}
\item[(i)] ({\em $BV^2$-* weak convergence}).  The sequence $$f_{n}^h(t)=2^n \int_{I_{n,k}} \GA^h_\tau(\tau)\, \d \tau  \,$$
satisfies
$$\foralls (n,k) \quad f^h_n(t) \overset{*-BV^2}{\rightharpoonup} f^\infty_n(t) \quad \foralls t\in [0,1]$$
as $h\rightarrow \infty$; 
\item[(ii)] ({\em Martingale convergence}). We have
$$\underset{n\rightarrow \infty}{\lim}\|f_n^\infty - \GA_t\|_{L^2([0,1],H^1(\Circ,\RR^2))}=0\,.$$
\end{itemize}
\end{defn}

Then, the proof of Theorem~\ref{local_existence} gives actually the following result:

\begin{thm}\label{comp-sigma} For every  $\gm_0,\gm_1\in\Bb$, the following properties hold.
\begin{itemize}
\item[(i)] Every bounded set of $\mathcal{P}(\gm_0,\gm_1)$ is sequentially compact with respect to $\sigma$-topology and the $\sigma$-convergence implies the strong convergence in $H^1([0,1], W^{1,1}(\Circ,\RR^2))$.
\item[(ii)] The energy $E$ is lower semicontinuous with respect to the $\sigma$-topology.
\end{itemize}
\end{thm}

\begin{rem}\label{existence-sigma}
The result of existence proved by Theorem \ref{local_existence} can  now be presented as follows. We can suppose that  $\{\GA^h\} \subset\mathcal{P}(\gm_0,\gm_1) $ with  $\underset{h}{\sup}\, E(\GA^h)< \infty$. 
Then $\{\GA^h\}$ is uniformly bounded in  $\mathcal{P}(\gm_0,\gm_1)$ and, by points $(i)$ and $(ii)$ of the previous theorem, energy $E$ reaches its minimum on  $\mathcal{P}(\gm_0,\gm_1)$.
\end{rem}

\subsection{Geodesic Distance Between Geometric Curves}
\label{13}

Theorem~\ref{local_existence} shows the existence of a geodesic between any two parameterized curves in $\Ba$. We are interested now in defining a geometric distance between geometric curves (i.e., up to reparameterization). To this end we consider the set of curves belonging to $\Bb$ that are globally injective and oriented counter-clockwise . Such a set is denoted by $\Bb^i$.

 The space of geometric curves is defined as $\Bb^i \, / \, {\rm Diff}^{BV^2}(\Circ)$. We remind that in this section ${\rm Diff}^{BV^2}(\Circ)$ denotes the set defined in Definition \ref{reparam}. For every $\gm\in \Ba^i$ its equivalence class (called also  geometric curve) is denoted by $[\gm]$.

\begin{simplified}
\begin{rem}[{\bf Reparameterization and $BV^2$-norm}]
 Remark that  for every $\gm\in \Ba$ and $\varphi \in \rm{Diff}^{BV^2}(\Circ)$ we have 
$$
\|(\gm\circ \varphi)'\|_{BV(x)}=\|\gm'\|_{BV(x)} \,.
$$
Moreover, 
 $$\|\gm\circ \varphi\|_{L^1(\Circ,\RR^2)}\leq \|\gm\circ\varphi\|_{L^\infty(\Circ,\RR^2)} = \|\gm\|_{L^\infty(\Circ,\RR^2)}\,,$$ 
 so by~\eqref{embedding}, we get
\begin{equation}\label{reparamet}
\|\gm\circ \varphi\|_{BV^2(x)}\leq 2\|\gm\|_{BV^2(x)} \,.
\end{equation}
\end{rem}
\end{simplified}

The next proposition defines a distance on the set of curves belonging to $\Ba^i$  up to reparameterization. 
 
\begin{prop}\label{distance-geom}
 The Procrustean dissimilarity measure defined by
\begin{equation}
\mathcal{D}([\gm_0] , [\gm_1]) = \inf\, \{ d(\gm_0 \circ \phi, \gm_1 \circ \psi ) \,: \,\phi,\psi \in {\rm{Diff}}^{BV^2}(\Circ)\}
\end{equation}
is a distance on the set of $\Ba^i$-curves up to reparameterization.
\end{prop}

\begin{proof}
Clearly, the function $\mathcal{D}$ is symmetric, nonnegative, and it is equal to zero if $[\gm_0] = [\gm_1]$. 
\par Note also  that, the distance $d$ is invariant under reparameterization, so that   
\begin{equation}\label{invariance-2}
d(\gm_0,\gm_1) = d(\gm_0 \circ \phi,\gm_1\circ \phi), \quad \foralls \phi \in {\rm{Diff}}^{BV^2}(\Circ)\,.
\end{equation}

Then from the invariance~\eqref{invariance-2} it follows
that for every $\varphi_1, \varphi_2,\varphi_3 \in  {\rm{Diff}}^{BV^2}(\Circ)$ and $\gm_1, \gm_2, \gm_3 \in \Ba$ we have 
$$
\begin{array}{ll}
d(\gm_1\circ\varphi_1,\gm_2\circ\varphi_2) & \leq d(\gm_1\circ\varphi_1,\gm_3\circ\varphi_3) + d(\gm_3\circ\varphi_3,\gm_2\circ\varphi_2)\\
& =  d(\gm_1\circ\varphi_1\circ\varphi_3^{-1},\gm_3) + d(\gm_3,\gm_2\circ\varphi_2\circ\varphi_3^{-1})
\end{array}
$$
which implies that the  triangle inequality is satisfied for $\mathcal{D}$. 
\par We prove now that $\mathcal{D}([\gm_0] , [\gm_1]) = 0$ implies that $[\gm_0] = [\gm_1]$. We assume that $\gm_0$ is parameterized by the constant speed parameterization (this  is possible because of the invariance of $d$ under reparameterization).
So there exists a sequence $\{\phi_h\}$ of reparameterizations such that 
$$d(\gm_0,\gm_1 \circ \phi_h) \leq \frac 1 h\,.$$
Then we can consider the sequence $\{\GA^h\}$ of  minimal geodesics joining $\gm_0$ and $\gm_1 \circ \phi_h$. 
Similarly to \eqref{bound0}, by setting $m = \inf_h \inf_t m_{\GA^h(t)} > 0$, from \eqref{flow}, it follows 
\begin{equation}\label{bound}\| \gm_0 - \gm_1  \circ \phi_h \|^2_{BV^2(\Circ,\RR^2)} \leq \int_0^1 \|\GA_t^h \|^2_{BV^2(\Circ,\RR^2)}(t) \d t \leq  \frac{1}{m^2h^2}\,.
\end{equation}
Now,  because of Lemma \ref{bound-rep-bv} the sequence $\{\varphi_h\}$ is bounded in $\BV^2(\Circ, \Circ)$ so it converges (up to a subsequence) to some $\varphi$ with respect to the weak-* topology. Thus, by taking the limit, we get  $\gm_0 = \gm_1 \circ \phi$.

\end{proof}

The next theorem proves an existence result for geodesics.

\begin{thm}[{\bf geometric existence}]\label{localbv2}
Let $\gm_0\, \gm_1 \in \Ba^i$ such that $\mathcal{D}([\gm_0] , [\gm_1])<\infty$. Then there exists a minimizer  of $\mathcal{D}([\gm_0] , [\gm_1])$. More precisely, there exists $\tilde{\gm} \in [\gm_1]$ and $\GA\in \mathcal{P}(\gm_0,\tilde{\gm})$ such that $E(\GA)=d^2(\gm_0,\tilde{\gm}) = \mathcal{D}^2([\gm_0] , [\gm_1])$.	
\end{thm}

\begin{proof} In the following we denote by $\gm_0$ and $\gm_1$ the parameterizations by the constant speed parameterization.
Because of the invariance~\eqref{invariance-2} we can write 
$$\mathcal{D}([\gm_0] , [\gm_1]) = \inf\, \{ d(\gm_0, \gm_1 \circ \psi ) \,: \,\psi \in {\rm{Diff}}^{BV^2}(\Circ)\}.$$
We consider a sequence $\{\psi_h\} \subset {\rm{Diff}}^{BV^2}(\Circ)$ such that $ d(\gm_0, \gm_1 \circ \psi_h )\rightarrow \mathcal{D}([\gm_0] , [\gm_1])$ and we suppose that $\sup_h d(\gm_0, \gm_1 \circ \psi_h ) < \infty$.  By Theorem \ref{local_existence}, for every $h$, there exists a geodesic  $\GA^h$  between  $\gm_0$ and $\gm_1\circ \psi_h$  such that $d^2(\gm_0, \gm_1 \circ \psi_h ) =E(\GA^h)$. 
We show that there exists  $\psi_\infty \in {\rm{Diff}}^{BV^2}(\Circ)$ such that $ \mathcal{D}([\gm_0] , [\gm_1]) = d(\gm_0, \gm_1 \circ \psi_\infty)$. 
\par By the same arguments used in the proof of Theorem~\ref{local_existence}, we can define (see Step 1) a path $\GA^\infty\in H^1([0,1],\Bb)$ such that 
\begin{equation}\label{convergence-path}\int_0^1 \GA^h_t(t)(s) \d t \rightarrow \int_0^1 \GA^\infty_t(t)(s) \d t\,\end{equation}
and (see Step 2)
\begin{equation}\label{convergence-path20}
\GA^h(t) \overset{W^{1,1}(\Circ,\RR^2)}{\longrightarrow} \GA^\infty(t) \quad \forall\,t\in [0,1]\,,
\end{equation}
\begin{equation}\label{convergence-path2}
\GA^h_t(t) \overset{W^{1,1}(\Circ,\RR^2)}{\longrightarrow} \GA^\infty_t(t) \quad \forall\,t\in [0,1].
\end{equation}
Remark that, because of \eqref{flow},  we have 
\begin{equation}\label{flow-path}
\int_0^1 \GA_t^h(t)(s)\d t = \gm_1\circ \psi_h(s)-\gm_0(s)\quad\quad  \forall \, s\in \Circ\,.
\end{equation}
Now, analogously to \eqref{bound}, this guarantees a bound on $\{\gm_1\circ \psi_h\}$ and, because of Lemma \ref{bound-rep-bv}, the sequence $\{\psi_h\}$ is bounded in $\BV^2(\Circ, \Circ)$ so it converges (up to a subsequence) to some $\psi_\infty$ with respect to the weak-* topology. Thus, we have  

 $$\gm_1\circ\psi_h\overset{W^{1,1}(\Circ,\RR^2)}{\rightharpoonup}\gm_1\circ\psi_\infty\,.$$
\par Now, as $\GA^h(0)=\gm_0$ and $\GA^h(1)= \gm_1\circ \psi_h$, from \eqref{convergence-path20}, it follows 
 $$\GA^\infty(0) = \gm_0\,,\quad \GA^\infty(1) = \gm_1\circ\psi_\infty\,$$
which imply that $\GA^\infty\in \mathcal{P}(\gm_0,\gm_1\circ\psi_\infty)$. Moreover,  denoting  by $\varphi_{\GA^h(t)}$ and $ \varphi_{\GA^\infty(t)}$ the constant speed parameterization of $\GA^h(t)$ and $\GA^\infty(t)$ respectively, by~\eqref{convergence-path20} and Lemma \ref{conv-w11}, $\varphi_{\GA^h(t)}$ converges to $ \varphi_{\GA^\infty(t)}$ in $W^{1,1}(\Circ,\Circ)$. Then, similarly to \eqref{sci}, we get  
$$\| \GA^{\infty}_t(t)\|_{BV^2(\GA^\infty(t))} 
\leq \underset{h\rightarrow \infty}{\liminf} \, \|\GA^{h}_t(t)\|_{BV^2(\GA^h(t))}\,.
$$
By integrating the previous inequality and using Fatou's lemma we get 
$$E(\GA^\infty)\leq \lim_{h\rightarrow \infty} E(\GA^h)=\lim_{h\rightarrow \infty} d^2(\gm_0,\gm_1\circ\psi_h)=\mathcal{D}^2([\gm_0] , [\gm_1]) \,,$$
which implies that $\mathcal{D}^2([\gm_0] , [\gm_1]) = d^2(\gm_0, \gm_1 \circ \psi_\infty)=E(\GA^\infty)$.
\par 
\end{proof}


\section{Geodesics in the Space of Sobolev Curves}\label{Hs}

In this section, we study the geodesic boundary value problem in the class of curves belonging to $\hs$ with $k\geq 2$ integer.

\begin{simplified}
 We recall that, by identifying $\Circ$ with $[0,1]$,  $u\in H^k(\Circ,\R)$ if $u$ is 1-periodic and 
$$\|u\|_{H^k(\Circ,\R)}^2 = \sum_{n\in \mathbb{Z}} (1+n^{2s})|\hat f(n)|^2 \,<\, \infty\,$$
with 
$$\hat f(n) = \int_{\Circ} f(s) e^{-\imath ns} \d s\,.$$
Moreover, $\|u\|_{H^k(\Circ,\R)}$ defines a norm and $H^k(\Circ,\R)$ equipped with such a norm is a Hilbert space. 
\end{simplified}
We remind the continuous embedding 
\begin{equation}\label{embedding1}
H^k(\Circ,\RR^2) \hookrightarrow C^{1}(\Circ,\RR^2) \quad \foralls \, k\geq 2\,.
\end{equation}



In this framework, the proof of  existence of geodesics is simpler because we can use the compactness properties of the Bochner space of paths with respect to the weak topology. 


 
We define the class of the parameterized $H^k$-curves as follows. 

\begin{defn}
[{\bf $H^k$-curves}]\label{Hscurves}
We define $\Ha$ as the class  of counterclockwise oriented curves belonging to $H^k( \Circ, \R^2)$ ($k\geq 2$ integer)  and such that for any $\gamma \in \Ha$, $\gamma'(s) \neq 0$ for all $s \in \Circ$. 
For every $\gm\in \Ha$ the constant speed parameterization $\varphi_\gm$ can be defined as in Remark \ref{arc-len}. 
\end{defn}

\begin{rem}[{\bf reparameterizations}]
In this section, we denote by 	${\rm Diff}_0^k(\Circ)$ the connected component of identity of the Sobolev diffeomorphisms group  $\varphi\in H^k(\Circ,\Circ)$ such that $\varphi^{-1} \in H^k(\Circ,\Circ)$.
As in the previous section, we will use weak topologies on the lifts of these diffeomorphisms and the following lemma. Finally, the chosen topology on ${\rm Diff}_0^k(\Circ)$ is the pullback topology of the Sobolev norm by the lift  $\varphi \mapsto \tilde{\varphi} - \Id \in H^k(\Circ,\R)$.
\end{rem}

\begin{lem}\label{BoundedReparametrizations}
Let $\varepsilon$ and $M$ be two positive real numbers. There exists a positive constant $D = D(\varepsilon,M)$ such that 
if $\gm \in B$ is such that $\| \gm'(s) \| \geq \varepsilon >0$ and $\| \gm \|_{H^k(\Circ,\R^2)} \leq M$, then the reparameterization $\phi_\gm$ satisfies
$ \| \phi_\gm \|_{H^k(\Circ,\Circ)} \leq D$. 
\end{lem}

\begin{proof}
Recall that the reparameterization $\phi_\gm$ is the inverse of $s_\gm$ defined  $s_\gm(s) =\frac{1}{\len(\gm)} \int_{s_0}^s \,|\gm'(t)| \,dt$,
where $s_0 \in \Circ$ is a chosen basepoint. We thus have $s_\gm'(x) > \frac{\varepsilon }{ \len(\gm)} $.  Using \eqref{embedding-sob} below, $\| \gm' \|$ is bounded by 
$C M$. Therefore, $x \in \R^2 \mapsto \| x \| \in \R$ being a smooth function on $B(0,\alpha M) \setminus B(0, \frac{\varepsilon }{ \len(\gm)} )$, there exists a constant $\beta$ such that $\| s_\gm \|_{H^k(\Circ,\Circ)} \leq \beta M$. 

 In Lemma 2.8 in \cite{Inci2013}, it is proven that the inversion on $\mathcal{D}^s(\R^d)$ is continuous and a locally bounded map when $s>d/2 +1$. Their proof actually shows that if $\phi \in \mathcal{D}^s$ is bounded in $H^s(\R^d,\R^d)$, namely, $\| \Id - \phi \|_{H^s(\R^d,\R^d)} \leq M'$ and $\det D\phi > \varepsilon'$ then $\| \Id - \phi^{-1} \|_{H^s(\R^d,\R^d)} \leq D'(M',\varepsilon')$.  The same proof would be valid in our situation replacing $\R^d$ by $\Circ$, however, we present a simple argument to apply their result.
Let $m:\R \to \R$ be a smooth map such that $m(x) = 1$ if $x \in [-2,2]$ and $m(x) = 0$ if $x \notin ]-3,3[$. Then the map $ \Psi:  {\rm Diff}_0(\Circ) \to \mathcal{D}^k(\R)$ defined by 
$\Psi(\phi) = \Id + m(\tilde{\phi} - \Id)$ satisfies 
\begin{equation}
c \| \tilde{\phi} - \Id \|_{H^k(\Circ,\R)} \leq \|\Psi(\phi) - \Id  \|_{H^k(\R,\R)} \leq c' \| \tilde{\phi} - \Id \|_{H^k(\Circ,\R)} \,,
\end{equation}
where $c,c'$ are some positive constants. The first inequality is clear while the second is obtained because $H^k(\R)$ is a Hilbert algebra since $k>1/2$.
Lemma 2.8 in \cite{Inci2013} implies that $\Psi(\phi)^{-1}$ is bounded in $\mathcal{D}^k$ and $\Psi(\phi)^{-1}$ is equal to $\Psi(\phi^{-1})$ on $[-2,2]$, which implies the result.
\end{proof}

From~\eqref{embedding1} it follows that  there exists a constant $C$ such that 
\begin{equation}\label{embedding-sob}
	\|\gm'\|_{L^\infty(\Circ,\R^2)}\leq   C\|\gm\|_{\hs} \quad \foralls \gm \in \Ha \,.
\end{equation}
Moreover, it is easy to verify that $\Ha$ is an open set of $\hs$. 


As in the previous section we define the Hilbert space $\hsg$ as the space $\hs$, where integration and derivation are performed with respect to the measure $\d \gm = |\gm'| \d s$. 
The tangent space at $\gm \in \Ha$ is endowed with the corresponding norm and is denoted by $H^k(\gm)$.
More explicitely, we have, denoting $\langle \cdot, \cdot \rangle$ the usual scalar product on $\R^2$,
\begin{equation}
\| f \|_{\hsg}^2 = \int_{\Circ} \langle f, f\rangle \, \d\gm + \int_{\Circ} \left\langle \frac{\d^k}{\d\gm^k}f, \frac{\d^k}{\d\gm^k}f \right\rangle\, \d\gm\,.
\end{equation}
This defines a smooth Riemannian metric on $\hs$ (see \cite{Bruveris}).

As in the previous section, for every $\gm_0,\gm_1\in \Ha$ we consider the class of paths $\GA \in H^1([0,1],\Ha)$ such that $\GA(0)=\gm_0$ and $\GA(1)=\gm_1$. The energy of a path is defined as  $$E(\GA)=\int_0^1 \|\GA_t(t)\|_{H^k(\GA(t))}^2\,\d t$$ and the geodesic distance $d(\gm_0,\gm_1)$ is defined accordingly (see Definition~\ref{defn-geodesic-paths}).

Moreover, Lemma 5.1 in \cite{Bruveris} proves the equivalence of the norms of $\hs$ and $\hsg$. The result states that, for every $\gm_0\in \Ha$, there exists a constant $C=C(\gm_0, D) > 0$ such that 
\begin{equation}\label{bound-sob-norms}
\frac 1 C \| f \|_{\hs} \leq  \| f \|_{\hsg} \leq C \| f \|_{\hs}\,
\end{equation}
for every $\gm \in \Ha$ such that $d(\gm_0,\gm)< D$. This proves in particular that the constant $C$ is uniformly bounded on every geodesic ball. 

Finally, in order to compare the $H^k$-norm after reparameterization, we remark that 
\begin{equation}\label{change-sob}
\begin{array}{c} 
\displaystyle{ 
\| f \|_{\hsg} = \| f\circ\varphi_\gm \|_{\hs}^{\gm}}\,, \vspace{0.3cm}\\
\displaystyle{(\| f \|_{\hs}^{\gm})^2 = \len(\gm)\| f \|^2_{\leb{2}} +  \len(\gm)^{1 - 2k}\|f^{(k)}\|^2_{\leb{2}} \quad \quad \forall \, f\in \hs.}
\end{array}
\end{equation}

We now prove an existence result for geodesics in the Sobolev framework.

\begin{thm}[{\bf existence}]\label{existence_sobolev}
Let $\gm_0, \gm_1 \in \Ha$ such that $d(\gm_0, \gm_1)<\infty$.
Then, there exists a geodesic between $\gm_0$ and $\gm_1$.
\end{thm}

\begin{proof} 
Let $\{\GA^h\}$ be a minimizing homotopy sequence such that $E(\GA^h) < D^2$, where $D >d(\gm_0,\gm_1) $. 
Because of \eqref{bound-sob-norms} there exists a positive constant $M  = C(\gm_0,D)^2$ such that  
\begin{equation}\label{lower-velocity0}
\int_0^1 \|\GA_t^h(t)\|^2_{\hs}\,\d t\leq \int_0^1 M \|\GA_t^h(t)\|^2_{\hsgt}\,\d t \leq  M E(\GA)\,.
\end{equation}
This implies  that  $\GA^h_t$ is uniformly bounded in  $L^2([0,1],\hs)$ and, because of the boundary conditions, that  $\GA^h$ is uniformly bounded in  $H^1([0,1],H^k(\Circ,\R^2))$. 

Therefore there exists a subsequence of $\{\GA^h\}$ that weakly converges in   $H^1([0,1],H^k(\Circ,\R^2))$.
Since the embedding 
\eq{
	H^1([0,1],H^k(\Circ,\R^2))\subset C([0,1], H^{k-1}(\Circ,\R^2))
}
is compact, there exists another subsequence (not relabeled)  that converges to a path $\GA^\infty$ belonging to $L^2([0,1], C(\Circ, \R^2))$,
$$\GA^h \rightarrow \GA^\infty \quad \mbox{in}\quad C([0,1],H^{k-1}(\Circ,\R^2))\,.$$
This proves in particular that 
$$ 
	\GA^h(t) \overset{W^{1,1}(\Circ, \RR^2)}{\rightarrow} \GA^\infty(t) \; \foralls t\in [0,1] \,
	\qandq
 	\GA^h_t(t) \overset{H^k(\Circ, \RR^2)}{\rightharpoonup} \GA^\infty_t(t) \; \foralls t\in [0,1] \,.
$$
Now, as   $\GA^h(t)$ converges in $W^{1,1}(\Circ, \RR^2)$ towards $\GA^\infty(t)$, the constant speed reparameterizations at time $t$, denoted respectively by  $\phi_{\GA^h(t)}$ and $\phi_{\GA^\infty(t)}$, also converge in $W^{1,1}(\Circ, \Circ)$.

Using Lemma \ref{BoundedReparametrizations}, we have that $\phi_{\GA^h(t)}$ are (uniformly w.r.t. $t$ and $h$) bounded in $H^k(\Circ,\Circ)$ so that by a direct adaptation of Lemma 2.7 in \cite{Inci2013} (or the argument developed in Lemma \ref{BoundedReparametrizations}), the sequence $\Ga_t^h \circ \phi_{\GA^h(t)}$ is (uniformly) bounded in $H^k(\Circ,\R^2)$.
Thus, by the same arguments used to prove~\eqref{sci2} we obtain weak convergence on a dense subset. Since the sequence is bounded, the weak convergence follows: 
$$\GA^h_t\circ \phi_{\GA^h(t)}\overset{\hs}{\rightharpoonup}\GA^\infty_t\circ\phi_{\GA^\infty(t)} \quad \foralls \,t\in [0,1]\,.$$

Now, because of \eqref{change-sob}, we have
$ \|\GA^{h}_t(t)\|_{\hsgt}= \| \GA^h_t \circ \phi^h (t)\|^{\GA^h(t)}_{\hs}$. Recall that, because of the strong convergence in $W^{1,1}(\Circ, \Circ)$ of the  constant speed parameterizations we have $\len(\GA^h(t))\rightarrow \len(\GA^\infty(t))$ for every $t$. 
Then, for every $t$, we get
$$
\begin{array}{lll}
	\| \GA^{\infty}_t(t)\|_{\hsgt} &=  \| \GA^\infty_t \circ \phi^\infty(t) \|^{\GA^{\infty}(t)}_{\hs}   
	&\leq \underset{h\rightarrow \infty}{\liminf}  \| \GA^h_t \circ \phi^h (t)\|^{\GA^h(t)}_{\hs}	\vspace{0.2cm} \\
&=\underset{h\rightarrow \infty}{\liminf} \|\GA^{h}_t(t)\|_{\hsgt}\,.&\\
\end{array}
$$

By integrating the previous inequality and using Fatou's lemma we get that $\GA^\infty$ minimizes $E$ and the theorem ensues.
\end{proof}

In ~\cite{Bruveris}, the authors prove (Theorem 1.1) that the space of immersed curves is geodesically complete with respect to the $H^k$-metrics ($k\geq 2$ integer). 
Since $H^k$-metrics ($k\geq 2$ integer) are smooth Riemannian metrics, minimizing geodesics are given locally by the exponential map.
Moreover, from Theorem \ref{existence_sobolev}, we have the existence of minimal geodesics (in our variational sense) between any two points. Therefore, the minimizing curve found by our variational approach coincides with an exponential ray. Thus, we have the following corollary.

\begin{cor}[\bf surjectivity of the exponential map] 
The exponential map on $\Ha$ for $k \geq 2$ integer is defined for all time and is surjective.
\end{cor}

The so-called Fr\'echet or K\"archer mean, often used in imaging \cite{Pennec}, is a particular case of minimizers of the distance to a closed subset. The surjectivity of the exponential map enables the use of~\cite[Theorem 3.5]{Azagra}, which proves that the projection onto a closed subset is unique on a dense subset.
A direct theoretical consequence of this surjectivity result is the following result. 

\begin{prop}
Let $k\geq 2$ be an integer.
For any integer $n\geq 1$, there exists a dense subset $D \subset \Ha^n$, such that the K\"archer mean associated with any $ (\gm_1,\ldots,\gm_n) \in D$, defined as a minimizer of 
\begin{equation}
	\min_{\gm \in \Ha} \;  \sum_{i=1}^n d(\gm , \gm_i)^2\,,
\end{equation}
is unique.
\end{prop} 

\begin{proof}
Let $S$ be the diagonal in $\Ha^n$. The set $S$ is a closed subset of $ {\underbrace{\Ha \times \ldots \times \Ha}_{n \text{ times}}}$. In~\cite[Theorem 3.5]{Azagra} the authors prove that the set of minimizers of $\arg \min_{y\in S} d(x,y)$ is a singleton for a dense subset in~$\Ha^n$.
\end{proof}

We call the space of geometric curves the quotient space $\Ha^i \, / \, {\rm{Diff}_0^k(\Circ)}$, where $\Ha^i$ denotes the set of globally injective curves. The notation $[\gm]$ represents the class of $\gm$ in the quotient space.
Analogously to the $BV^2$-case we can define a distance $\mathcal{D}$ between two geometric curves (see Proposition \ref{distance-geom}). The fact that the distance satisfies $d([\gm_0],[\gm_1]) = 0$ implies $[\gm_0]=[\gm_1]$ can be proven using Lemma \ref{BoundedReparametrizations}. This lemma actually implies the following proposition.
\begin{prop}
Let $\gm \in \Ha^i$ and $r>0$. Denoting the equivalence class of $\gm$ by $[\gm]$ and $B_d(\gm,r)$ the closed geodesic ball of radius $r$,
 the set $[\gm] \cap B_d(\gm,r)$ is weakly compact. 
 %
\end{prop}
\begin{proof}
Without loss of generality, we can assume that $\gm$ is parameterized by constant speed parameterization. 
There exist $\varepsilon$ and $M$, two positive constants, such that for any $\gm_1 \in B(\gm)$, we have $\| \gm_1' \|\geq \varepsilon$ and $\| \gm_1 \|_{H^s(\Circ,\R^2)} \leq M$. Let $\gm_n \in [\gm] \cap B_d(\gm,r)$ be a sequence then
using Lemma \ref{BoundedReparametrizations}, the reparameterizations $\phi_n$ are bounded in $H^k(\Circ,\Circ)$ and thus there exists a subsequence that weakly converges to $\phi_\infty \in H^k(\Circ,\Circ)$. It implies that $\phi_n'$ weakly converges to $\phi_\infty' \in H^{k-1}(\Circ,\R)$ and thus $\phi'_\infty  \geq \varepsilon$ since $k-1 >1/2$. Therefore $\phi_\infty$ belongs to ${\rm{Diff}}_0^k(\Circ)$. Using the same argument $\gm_n \circ \phi_n $ weakly converges to $\gm_\infty \circ \phi_\infty$. However, by definition we have $\gm_n \circ \phi_n = \gm$ so that $\gm_\infty \circ \phi_\infty = \gm$, which gives the result.
\end{proof}

By the same arguments used to prove Theorem \ref{localbv2} and the previous proposition, we easily get the following.

\begin{thm}[{\bf geometric existence}]\label{geom-sob}
Let $\gm_0\, ,\gm_1 \in \Ha^i$ such that $\mathcal{D}([\gm_0] , [\gm_1])<\infty$. Then there exists a minimizer  of $\mathcal{D}([\gm_0] , [\gm_1])$. More precisely, there exists $\gm_2 \in [\gm_1]$ such that $d(\gm_0,\gm_2) = \mathcal{D}([\gm_0] , [\gm_1])$.	
\end{thm}


\section{Numerical Computations of Geodesics}
\label{discretization}

In this section we discretize and relax problem~\eqref{problem} in order to approximate numerically  geodesic paths. Note that, since we use a gradient descent to minimize a discretized and relaxed energy, the resulting optimal discrete homotopy aims at approximating stationary points of the geodesic energy, and that these homotopies cannot be guaranteed to be globally minimizing geodesics.

\subsection{Penalized Boundary Condition at $t=1$}

To obtain a simple numerical scheme, we first relax the constraint $\GA(1)=\gm_1$ by adding at the energy $E$ a data fidelity term $H(\GA(1), \gm_1)$ taking into account the deviation between $\GA(1)$ and $\gm_1$. In the following we make use of the $H$ functional defined  in~\cite[equation 5.4]{Rigid-evol} .

Such a functional is defined as the following distance between two curves

$$H(\gm,\la) = \int_{\Circ} \int_{\Circ} \|\n_{\gm}(s) - \n_{\lambda}(t)\|^2 \; 
	k\left( \gm(s) , \lambda(t) \right)\; \d \gm(s) \d \lambda(t)\; \quad 
	\forall \gm, \la \in \Bb$$
where

$$	k(v, w) =  e^{-\frac{\|v-w\|^2}{2\sigma^2}} + e^{-\frac{\|v-w\|^2}{2\delta^2} } 
	\quad \quad \forall v,\,w\in \R^2\,.
$$
Here $(\sigma,\delta)$ are positive constants that should be adapted depending on the targeted application. We use a sum of Gaussian kernels to better
capture geometric features at different scales in the curves to be matched. This has been shown to be quite
efficient in practice in a different context in \cite{FX-ref}. According to our numerical tests, the use of more than two kernels does not improve the results.

This functional $H$ was initially proposed in~\cite{currents-matching} as a norm on a reproducing Hilbert space of currents.
 It can be shown to be a metric on the space of geometric curves, which explains why it is a good candidate to enforce approximately the boundary constraint at time $t=1$. We recall that $H$ is continuous with respect to strong topology of $W^{1,1}(\Circ,\RR^2)$. We refer to~\cite{Rigid-evol} for its properties and its discretization using finite elements. 

Then, given two curves $\gm_0, \gm_1\in \Bb$, we consider the following  problem 
\begin{equation}\label{simulation}
\begin{array}{c}
	\min \,\{F(\GA)\,:\, \GA\in H^1([0,1],\Bb)\,,\; \GA(0)=\gm_0\}\,,\vspace{0.2cm}\\
	F(\GA) = H(\GA(1), \gm_1) + E(\GA).
\end{array}
\end{equation}
To allow for more flexibility in the numerical experiments, we introduce a weighted $BV^2$-norm in the definition~\eqref{energy-bv2} of the energy $E$. Given some positive weights $(\la_0,\la_1,\la_2) \in (\RR^+)^3$,  we consider in this section 
\eql{\label{eq-E-BV2}
	E(\GA) = \int_0^1 \|\GA_t(t)\|_{BV^2(\GA(t))} \d t\,,
}
where, for all $\gm \in \Bb$ and $v \in T_\gm \Bb$,  
\eq{
	\|v\|_{BV^2(\gm)} = 
	\int_{\Circ}\left( 
		\la_0 |v(s)|
		+
		\la_1 \left|\frac{\d v}{\d \gm}(s) \right|
		+
		\la_2 \left|\frac{\d^2 v}{\d \gm^2}(s) \right| 
	\right)\d\gm(s). 
}

\subsection{Regularized Problem }
\label{sec-regul-problem-gamma}

The energy minimized in~\eqref{simulation} is both non-smooth and non-convex. In order to compute stationary points using a gradient descent scheme, we further relax this problem by smoothing the $\RR^2$-norm used to calculate the $BV^2$-norm. This approach is justified by the result proved  in Theorem~\ref{thm-gamma-cv}. 

The energy $E$ is regularized as
\eql{\label{eq-Eeps-def}
	E_\epsilon(\GA) = \int_0^1 \|\GA_t(t)\|_{BV^2(\GA(t))}^\epsilon \d t\,,
}
where $\epsilon>0$ controls the amount of smoothing, and the smoothed $BV^2$-norm is defined, for $\gm \in \Bb$ and $v \in T_\gm \Bb$, as
$$	
	\|v\|_{BV^2(\gm)}^\epsilon = \int_{\Circ}
		\left( 
			\la_0 \|v(s)\|_\epsilon + 
			\la_1 \left\|\frac{\d v}{\d \gm}(s)\right\|_\epsilon 
		\right)\|\gm'(s)\|_\epsilon\d s + 
		\la_2 TV_{\gm}^2(v), 
$$
where 
$$
	\forall x\in \RR^2, \quad
	\|x\|_\epsilon = \sqrt{\|x\|^2 + \epsilon^2}.
$$ 
A regularization of the second total variation is given by \eqref{reb-tv2} in the case of the finite element space.
The initial problem~\eqref{simulation} is then replaced by 
\begin{equation}\label{simulation-reg}
\begin{array}{c}
	\min \,\{F_\epsilon(\GA)\,:\, \GA\in H^1([0,1],\Bb)\,,\; \GA(0)=\gm_0\}\,,\vspace{0.2cm}\\
F_\epsilon(\GA) = H(\GA(1), \gm_1) + E_\epsilon(\GA)\,.
\end{array}
\end{equation}
This smoothing approach is justified by the following theorem.

\begin{thm}\label{thm-gamma-cv}
	Let $\gm_0\in \Bb$  and  $X= \{ \GA\in H^1([0,1],\Bb)\;:\; \GA(0)=\gm_0\}$.
 Then 
$$
	\underset{\epsilon\rightarrow 0}{\lim} \; \underset{\GA\in X}{\min}\; F_\epsilon(\GA) =  \underset{\GA\in X}{\min}\; F(\GA) \,.
$$
Moreover if $\{\GA_\epsilon\}$ is a sequence of minimizers of $F_\epsilon$ then there exists a subsequence (not relabeled) such that $\GA_\epsilon\overset{\sigma}{\rightarrow}\GA$ as $\epsilon\rightarrow 0$ (see Definition~\ref{sigma}) and $\GA$ is a minimizer of $F$.
\end{thm}

\begin{proof} We suppose without loss of generality  that $F$ and $F_\epsilon$ are not equal to infinity. 
Then, by  Theorem~\ref{comp-sigma} and Remark \ref{existence-sigma},  $F$ and $F_\epsilon$ reach their minima on $X$.

As $\{F_\epsilon\}_\epsilon$ is a decreasing sequence converging to $F$ pointwise, we get 
\begin{equation}\label{min-conv}
\underset{\epsilon\rightarrow 0}{\lim} \; \underset{\GA\in X}{\min}\;F_\epsilon(\GA) = \underset{\epsilon>0}{\inf}\; \underset{\GA\in X}{\min}\; F_\epsilon(\GA) =  \underset{\GA\in X}{\min}\; \underset{\epsilon>0}{\inf}\; F_\epsilon(\GA) =  \underset{\GA\in X}{\min}\; F(\GA) \,.
\end{equation}

Thus, if $\{\GA_\epsilon\}$ is a sequence of minimizers of $F_\epsilon$ we have $F(\GA_\epsilon)< F_\epsilon(\GA_\epsilon)$ so that $\{\GA_\epsilon\}$ is a minimizing sequence for $F$.  Then, by Theorem~\ref{comp-sigma}, there exists $\GA\in X$ and  a subsequence (not relabeled) such that $\GA_\epsilon\overset{\sigma}{\rightarrow}\GA$ as $\epsilon\rightarrow 0$. As $F$ is lower semicontinuous with respect to the $\sigma$ convergence, from~\eqref{min-conv},  it follows that $\GA$ is a minimizer of $F$.
\end{proof}

\subsection{Finite Element Space}
\label{sec-finite-elements}

In the following, to ease the notation, we identify $\RR^2$ with $\CC$ and $\Circ$ with $[0,1]$ using periodic boundary conditions.  

To approximate numerically stationary points of~\eqref{simulation-reg}, we discretize this problem by using finite element approximations of the homotopies, which are piecewise linear along the $s$ variable and piecewise constant along the $t$ variable. This choice of elements is justified by the fact that the evaluation of the energy requires the use of two derivatives along the $s$ variable, and a single derivative along the $t$ variable. 

\paragraph{Finite elements curves.}

A piecewise affine curve with $n$ nodes is defined as
\eq{
	\foralls s \in [0,1], \quad
	\gm(s) = \sum_{j=1}^n \tilde \gm_j \xi_j(s)\,,
}
where we used piecewise affine finite elements
\begin{align*}
	\xi_j(s) &= \max\left\{0, 1- n\left| s- \frac{j}{n} \right|\right\} \quad  s\in [0,1],\;  
		\foralls j= 1, \ldots, n-1 \,,\\
	\xi_n(s) &= \max\left\{0, 1- n\left|s \right|\right\} + \;\max\left\{0, 1- n\left| s- 1 \right|\right\},  \quad s\in [0,1].
\end{align*}
Here, $\tilde\gm \in \CC^{n}$ denotes the coordinates of $\gm$ and we denote $\gm=P_{1}(\tilde \gm)$ the corresponding bijection. 

\paragraph{Finite elements homotopies.}

We consider the finite dimensional space of homotopies of the form
\eql{\label{eq-fem-homotop}
	\forall (t,s) \in [0,1]^2, \quad
	\GA(t)(s) = \sum_{i=1}^N \sum_{j=1}^n \tilde\GA_{i,j} \zeta_i(t) \xi_j(s)\,,
}
where we used piecewise constant finite elements
\begin{align*}
	\zeta_i(t) =\myI{}_{[\frac{i}{N},\frac{i+1}{N}]}(t) \quad  \foralls i = 1, ..., N-1 \,,\quad 
	\zeta_N(t) =\myI{}_{[0,\frac{1}{N}]}(t) \,.
\end{align*}
Here, $\tilde\GA \in \CC^{N \times n}$ denotes the coordinates of $\GA$ and we denote $\GA=P_{0,1}(\tilde \GA)$ the corresponding bijection.

\subsection{Discretized Energies}
\label{sec-discretized-energy}

The initial infinite dimensional optimization problem~\eqref{simulation-reg} is discretized by restricting the minimization to the finite element space described by~\eqref{eq-fem-homotop} as follows 
\begin{equation}\label{eq-optim-discrete}:
\begin{array}{c}
	\min \; \enscond{
		 \Ff_\epsilon(\tilde\GA)
	}{
		 \tilde\GA \in \CC^{N \times n}, \; 
		 \tilde\GA_{1,\cdot} = \tilde\gm_0
	}\,,
	\qwhereq	
	\Ff_\epsilon(\tilde\GA) = F_\epsilon(\GA)\,,
\end{array}
\end{equation}
where $\GA=P_{0,1}(\tilde \GA)$ and where the input boundary curves are $\gm_0 = P_1(\tilde\gm_0), \gm_1 = P_1(\tilde\gm_1)$, which are assumed to be piecewise affine finite elements. We have denoted here $\GA_{i,\cdot} = (\GA_{i,j})_{j=1}^n \in \RR^n$.

In order to ease the computation of gradients, we note that the energy $\Ff_\epsilon$ can be decomposed as
\eql{\label{eq-defn-Ff-eps}
	\Ff_\epsilon(\tilde\GA) = H( P_1(\tilde\GA_{N,\cdot}), \gm_1  ) 
		+ \Ee_\epsilon(\tilde\GA)\,,
	\qwhereq
	\Ee_\epsilon(\tilde\GA) = 
	\frac{1}{N-1} \sum_{i=1}^{N-1}J( \tilde\GA_{i,\cdot}, \tilde v_i )	\,,		
}
where we denoted the discrete time derivative vector field as
\eq{
	\tilde v_i = \frac{
			\tilde\GA_{i+1,\cdot}-\tilde\GA_{i,\cdot}
		}{N-1} \in \CC^n.
}
For $\tilde \gm \in \CC^n$ and $\tilde v \in \CC^n$, we used the notation
\eq{
	J(\tilde \gm, \tilde v) =  \sum_{\ell=0}^2 \la_\ell J_\ell(\tilde \gm, \tilde v)
}
and we define below the explicit computation of the terms $J_\ell(\tilde \gm, \tilde v)$ for $\ell=0, 1, 2$.

\paragraph{Zero order energy term ($\ell=0$).}

The $L^1$ norm of a piecewise affine field $v = P_1(\tilde v)$ tangent to a piecewise affine curve $\gm = P_1(\tilde\gm)$ can be computed as 
$$
	\int_{\Circ} |v(s)|_\epsilon |\gm'(s)|_\epsilon\d s = 
	\sum_{i=1}^n n|\Delta^+(\tilde{\gm})_i|_{\frac{\epsilon}{n}}\int_{\frac i n}^{\frac{i+1}{n}}|\tv_i\xi_i(s) + \tv_{i+1}\xi_{i+1}(s)|_\epsilon \d s.
$$
This quantity cannot be evaluated in closed form. For numerical simplicity, we thus approximate the integral by the trapezoidal rule. With a slight abuse of notation (this is only an approximate equality),  we define the discrete $L^1$-norm as
$$	
	J_0(\tilde{\gm},\tilde{v}) = \frac{1}{2} 
	\sum_{i=1}^n |\Delta^+(\tilde{\gm})_i|_{\frac{\epsilon}{n}}
		\Big(|\tv_i|_{\epsilon}+|\tv_{i+1}|_{\epsilon}\Big),
$$
where we used the following forward finite difference operator
\begin{equation*}
	\Delta^+: \CC^n\rightarrow  \CC^n \;, \quad \Delta^+( \tilde\gm )_i = \tilde\gm_{i+1} - \tilde\gm_i\,.
\end{equation*}

\paragraph{First order energy term ($\ell=1$).}

We point out that 
\begin{equation}\label{first-discrete}
\frac{\d v}{\d \gm} = \sum_{i=1}^n  \frac{\Delta^+(\tv)_i}{|\Delta^+(\tilde{\gm})_i|_{\frac{\epsilon}{n}}}\zeta_i 
\end{equation}
which implies that 

$$
	\int_{\Circ }\left|\frac{\d v}{\d \gm(s)}\right|_\epsilon|\gm'(s)|_\epsilon\d s
	= \sum_{i=1}^n \int_{\frac i n}^{\frac{i+1}{n}}n|\Delta^+(\tv)_i|_{\frac{\epsilon}{n}} \d s.
$$
Then the discretized $L^1$-norm of the first derivative is defined by
$$
	J_1(\tilde{\gm},\tilde{v})= \sum_{i=1}^n |\Delta^+(\tv)_i|_{\frac{\epsilon}{n}}.
$$

\paragraph{Second order energy term ($\ell=2$).}

As the first derivative is piecewise constant, the second variation coincides with the sum of the jumps of the first derivative. 
In fact, for every $g\in {\rm{C}}^1_c(\Circ, \RR^2)$, we have
$$
\begin{array}{ll}
\displaystyle{\int_{\Circ} \langle\frac{\d v}{\d \gm(s)}, \frac{\d g}{\d \gm(s)} \rangle\d \gm(s)} & =  
\displaystyle{\sum_{i=1}^n \int_{\frac i n}^{\frac{i+1}{n}} \langle\frac{\d v}{\d \gm(s)}, g'(s) \rangle\d s =}\\
& \displaystyle{= \sum_{i=1}^n  \left\langle \frac{\d v}{\d \gm}\left( \frac{i-1}{n}\right) - \frac{\d v}{\d \gm}\left( \frac i n\right) , g\left( \frac i n \right)\right\rangle\,.}
\end{array}
$$
Then, by~\eqref{first-discrete}, the second variation $TV_{\gm}^2(v)$ can be defined as  
\begin{equation}\label{reb-tv2}
	J_2(\tilde{\gm},\tilde{v}) = 
	\sum_{i=1}^n \left|\frac{\Delta^+(\tv)_{i+1}}{|\Delta^+(\tilde{\gm})_{i+1}|_{\frac{\epsilon}{n}}} -\frac{\Delta^+(\tv)_i}{|\Delta^+(\tilde{\gm})_i|_{\frac{\epsilon}{n}}}\right|_\epsilon. 
\end{equation}
We point out that $J_2$ represents a regularized definition of the second total variation because we evaluate the jumps by the smoothed norm  $|\cdot|_\varepsilon$.

\subsection{Minimization with Gradient Descent} 

The finite  problem~\eqref{eq-optim-discrete} is an unconstrained optimization on the variable $(\tilde \GA_{2,\cdot},\ldots,\tilde \GA_{N,\cdot})$, since $\tilde \GA_{1,\cdot}=\tilde \gm_0$ is fixed. The function $\mathcal{F}_\epsilon$ being minimized is $C^1$ with a Lipschitz gradient, and we thus make use of a gradient descent method. In the following, we compute the  gradient for the canonical inner product in $\CC^{N \times n}$.

Starting from some $\tilde\Ga^{(0)} \in \CC^{N \times n}$, we iterate 
\eql{\label{eq-grad-desc}
	\tilde\Ga^{(k+1)} = \tilde\Ga^{(k)} - \tau_k \nabla \Ff_\epsilon(\tilde\GA^{(k)}) \,,
}
where $ \tau_k>0$ is the descent step. A small enough gradient step size (or an adaptive line search strategy) ensures that the iterates converge toward a stationary point $\Ga^{(\infty)}$ of $\Ff_\epsilon$. 

The gradient $\nabla \Ff_\epsilon(\tilde\GA)$ is given by its partial derivatives as, for $i=2, \ldots ,N-1$, 
$$
	\partial_{\tilde\GA_i} \Ff_\epsilon(\tilde\GA) = 
	\frac{1}{N-1} \Big(
		\partial_1 J(\tilde{\GA}_i, \tilde{v}_i) - 
		\frac{1}{N-1}\partial_2 J(\tilde{\GA}_{i+1}, \tilde{v}_{i+1}) +  
		\frac{1}{N-1}\partial_2 J(\tilde{\GA}_{i-1}, \tilde{v}_{i-1})
		\Big)  \,,
$$ 
where $\partial_1 J$ (reap., $\partial_2 J$) is the derivative of $J$ with respect to the first (resp. second) variable and 
$$
	\partial_{\tilde{\GA}_N} \Ff_\epsilon = 
	\de + \frac{1}{(N-1)^2}\partial_1 J(\tilde{\GA}_{N-1}, \tilde{v}_{N-1})\,, 
$$
where $\de$ is the gradient of the map $\tilde \gm \mapsto H( P_1(\tilde \gm), \gm_1  )$ at $\tilde\gm=\tilde\GA_{N,\cdot}$. This gradient can be computed as detailed in~\cite{Rigid-evol}.

\subsection{Numerical Results}

In this section we show some numerical examples of computations of  stationary points $\tilde\Ga^{(\infty)}$ of the problem~\eqref{eq-optim-discrete} that is intended to approximate geodesics for the $BV^2$- metric. For the numerical simulations we define the $BV^2$-geodesic energy 
\eql{\label{num-bv2}
	E(\GA) = \int_0^1 \|\GA_t(t)\|_{BV^2(\GA(t))}^2 \d t\,
}
by the following weighted $BV^2$-norm
$$\|v\|_{BV^2(\gm)} = \int_{\Circ}\left( \mu_0\|v\|+ \mu_1 \left\|\frac{\d v}{\d \gm}(s)\right\|\right)\,\d\gm(s) + \mu_2 TV_\gm^2\left(v\right) \quad\forall\, v\in BV^2(\gm)\,,$$
where the parameters $(\mu_0,\mu_1,\mu_2)\in (\R^+)^3$  can be tuned for each particular application.
 
 We use a similar approach to approximate geodesics for the $H^k$-metric, for $k=2$, by replacing $E$ in~\eqref{num-bv2} by
$$
	E(\GA) = \int_0^1 \|\GA_t(t)\|_{H^2(\GA(t))}^2 \d t\,,
$$
where, for all $\gm \in H^s(\Circ,\RR^2)$ and $v \in H^s(\gm)$,  
\eq{
	\|v\|_{H^2(\gm)} = 
	\int_{\Circ}\left( 
		\lambda_0 \|v(s)\|^2
		+
		\lambda_1 \left\|\frac{\d v}{\d \gm}(s) \right\|^2
		+
		\lambda_2 \left\|\frac{\d^2 v}{\d \gm^2}(s) \right\|^2 
	\right)\d\gm(s)\,.
}
Note that, in contrast to the $BV^2$ case, this Sobolev energy is a smooth functional, and one does not need to perform a regularization~\eqref{eq-Eeps-def}, or equivalently, one can use $\epsilon=0$ in this case. We do not detail the computation of the gradient of the discretized version of the functional for the Sobolev metric, since these computations are very similar to the $BV^2$ case.

In the following experiments, we use a discretization grid of size $(N,n)=(10,256)$. The weights are set to $(\la_0,\la_1,\la_2)=(1,0,1)$ and $(\mu_0,\mu_1,\mu_2)=(1,0,1)$ (the curves are normalized to fit in $[0,1]^2$). These experiments can be seen as toy model illustrations for the shape registration problem, where one seeks for a meaningful bijection between two geometric curves parameterized by $\gm_0$ and $\gm_1$. Note that the energies being minimized are highly non-convex, so that the initialization $\GA^{(0)}$ of the gradient descent~\eqref{eq-grad-desc} plays a crucial role. 


\begin{figure}[h!]
\centering
\begin{tabular}{@{}c@{\hspace{1mm}}c@{}}
\includegraphics[width=.23\linewidth]{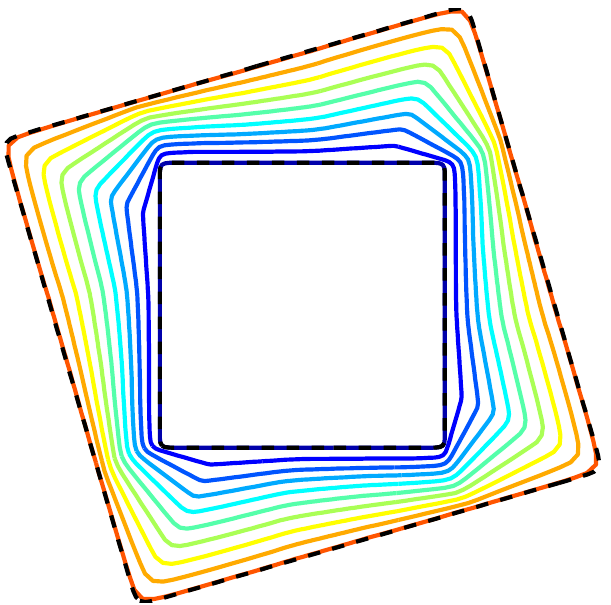}&
\includegraphics[width=.23\linewidth]{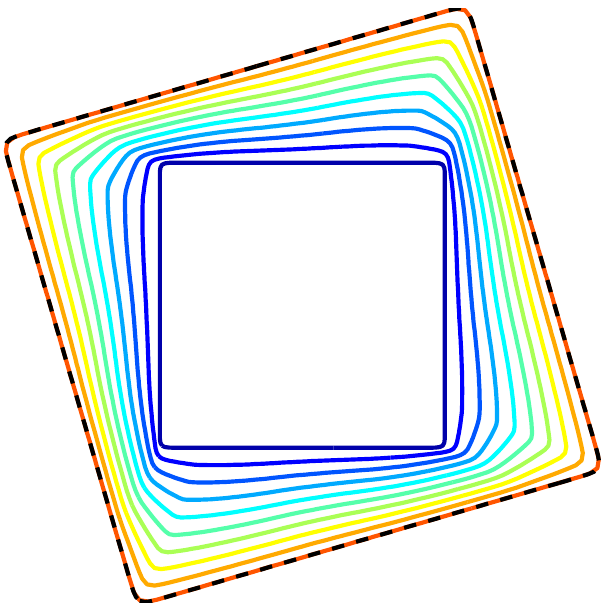}\\
\includegraphics[width=.23\linewidth]{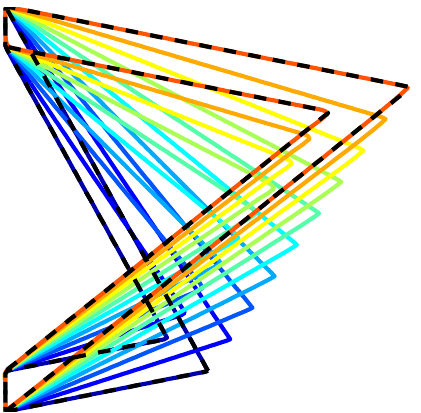}&
\includegraphics[width=.23\linewidth]{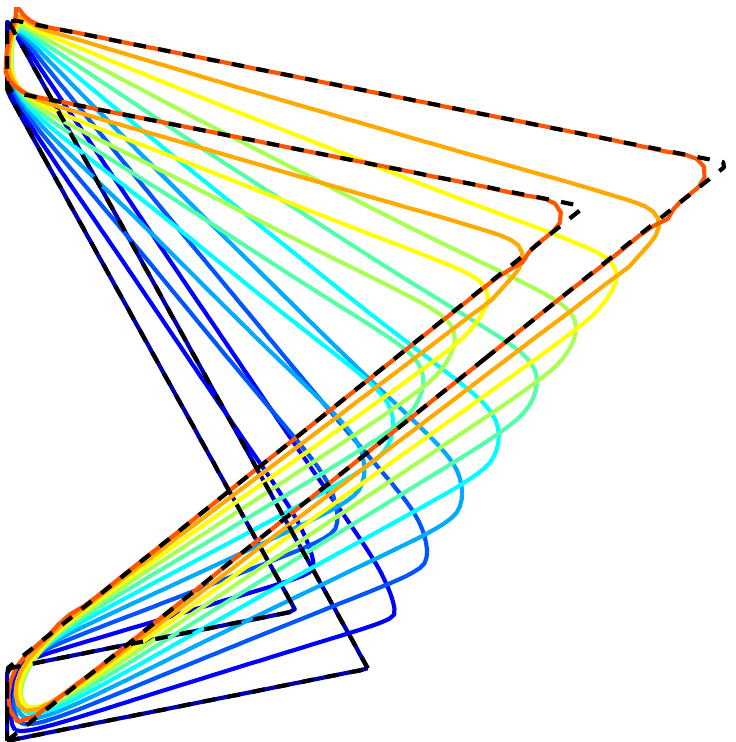}\\
$BV^2$ & Sobolev 
\end{tabular}
\caption{\label{geodesics} Homotopies $\Ga^{(\infty)}$ obtained for $BV^2$ Finsler energy (left) and Sobolev metric (right). Each image displays the initial curve $\gm_0$ (black one) and $\gm_1$ (dash line) and the optimal $\{\tilde\GA_{i,\cdot}\}_i$ where the index $1 \leq i \leq N$ is indicated by color variations between blue ($i=1$) and red ($i=N$). \vspace{1mm}}
\end{figure}

Fig.~\ref{geodesics}, top row, shows a simple test case, for which using a trivial constant initialization $\tilde\GA^{(0)}_i = \tilde\gm_0$, for both $BV^2$-and $H^2$-metric, produces a valid homotopy $\Ga^{(\infty)}$ between $\gm_0$ and $\gm_1$.  One can observe that while both homotopies are similar, the Sobolev metric produces a slightly smoother evolution of curves. This is to be expected, since piecewise affine curves are not in the Sobolev space $H^2(\Circ,\RR^2)$. 

Fig.~\ref{geodesics}, bottom row, shows a more complicated case, where using a trivial initialization $\tilde\GA^{(0)}$ fails to give a correct result $\tilde\GA^{(\infty)}$, because the gradient descent is trapped in a poor local minimum. We thus use as initialization the natural bijection $\tilde\GA^{(0)}$, which is piecewise affine and correctly links the singular points of the curves $\gm_0$ and $\gm_1$. It turns out that this homotopy is a stationary point of the energy~\eqref{eq-optim-discrete}, which can be checked on the numerical results obtained by the gradient descent. On the contrary, the Sobolev metric finds a slightly different homotopy, which leads to smoother intermediate curves.

In Fig. \ref{lastone} we show the influence of the choice of  $(\mu_0,\mu_1,\mu_2)$ in \eqref{num-bv2} on the optimization result. We point out in particular the role of $\mu_2$ which controls the jumps of the derivative and is responsible of the smoothness of the evolution.

\begin{figure}[h!]
\centering
\begin{tabular}{@{}c@{\hspace{1mm}}c@{\hspace{1mm}}c@{}}
\includegraphics[width=.25\linewidth]{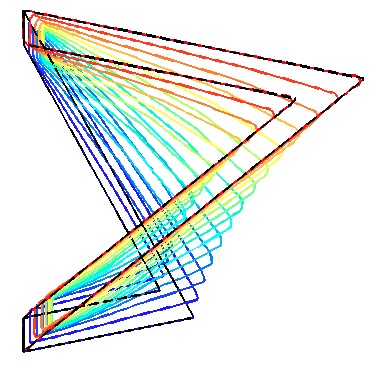}&
\includegraphics[width=.25\linewidth]{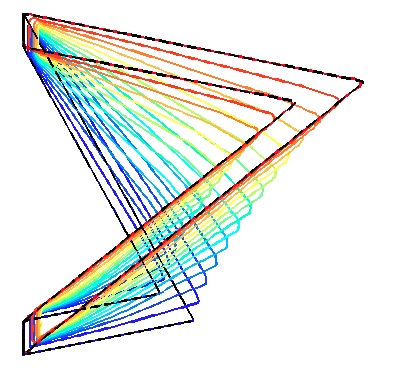}&
\includegraphics[width=.25\linewidth]{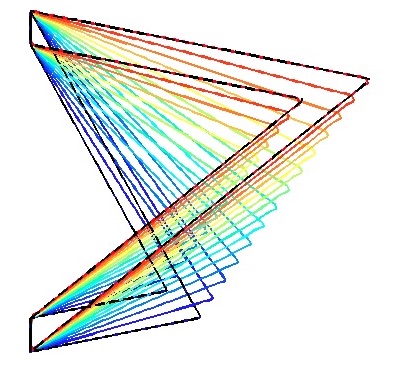}\\
(1,1,10e-05)&(1,1,0.001)&(1,1,0.01)\\
\end{tabular}
\caption{\label{lastone} Finsler $BV^2$-geodesics for different choices of $(\mu_0,\mu_1,\mu_2)$. \vspace{5mm}}
\end{figure}


\section{Conclusion}

The variational approach defined in this work  is a general strategy to prove the existence of minimal geodesics with respect to Finslerian metrics. 

In order to generalize previous results to more general Banach spaces, we point out the main properties which must be satisfied by the Banach topology:

\begin{itemize}
\item[(i)] the two constants $m_\gm, M_\gm$ appearing in Proposition \ref{equiv-norms} must be bounded on geodesic balls;

\item[(ii)] the topology of the space must imply a suitable convergence of the reparameterizations in order to get semicontinuity of the norm of $\{\GA^{h}_t \circ \phi^h(t) \}_h$; in the $BV^2$ case such a convergence is given by the  $W^{1,1}$-strong topology.
\end{itemize}

For the $BV^2$ metric, the major difficulty concerns the characterization of the weak topology of the space of the paths. The usual characterization of the dual of Bochner spaces $H^1([0,1],B)$ ($B$ is a Banach space) requires that the dual of $B$ verifies the RNP \cite{Bochner_dual, Bochner_dual_2}.
We point out that the martingale argument used to prove Theorem \ref{local_existence} avoids using such a characterization and allows one to define a suitable topology in such a space guaranteeing the lower semicontinuity of the geodesic energy.

Moreover, as pointed out in the introduction, the necessary conditions proved in ~\cite{Mord} are not valid in our case. This represents an interesting direction of research because optimality conditions allow one to study regularity properties of minimal geodesics. It remains an open question whether the generalized Euler-Lagrange equations in \cite{Mord} can be generalized to our case and give the Hamiltonian geodesic equations. Strongly linked to this question is the issue of convergence of the numerical method. Indeed, the convergence of the sequence of discretized problems would imply the existence of geodesic equations.

From a numerical point of view, as we have shown, the  geodesic energy suffers from many poor local minima. To avoid some of these poor local minima, it is possible to modify the metric to take into account some prior on the set of deformations. For instance, in the spirit of~\cite{Rigid-evol}, a Finsler metric can be designed to favor piecewise rigid motions.

\section*{Acknowledgement}

The authors want to warmly thank Martins Bruveris for pointing out the geodesic completeness of the Sobolev metric on immersed plane curves. This work has been supported by the European Research Council (ERC project SIGMA-Vision).

\newpage
\bibliographystyle{plain}  
\bibliography{bibliography}	

\end{document}